\documentclass[12pt]{article}
\usepackage[utf8]{inputenc}
\usepackage[letterpaper,top=2cm,bottom=2cm,left=3cm,right=3cm,marginparwidth=1.75cm]{geometry}
\usepackage{amsthm,amsfonts,pifont}
\usepackage{amsmath}
\usepackage{graphicx}
\usepackage[colorlinks=true, allcolors=blue]{hyperref}
\usepackage{fancyhdr}
\usepackage{comment}
\usepackage{xcolor}
\usepackage{enumerate}
\usepackage[toc]{appendix}

\allowdisplaybreaks

\newcounter{hypo}
\renewcommand{\thehypo}{H\arabic{hypo}}

\newenvironment{hypothesis}
  {\refstepcounter{hypo}
  \equation\tag{\thehypo}}
  {\endequation}

\newtheorem{teo}{Theorem}[section] 
\newtheorem{prop}{Proposition}[section] 
\newtheorem{lem}{Lemma}[section] 
\newtheorem{lemma}{Lemma}[section]
\newtheorem{defi}{Definition}[section] 
\newtheorem{rem}{Remark}[section]

\title{Existence, Stability and Controllability of the parabolic-parabolic thermistor model}
\author{Miguel R. Nuñez-Chávez \thanks{Departamento de Matemática, Universidade Federal do Mato Grosso, Cuiabá - MT, Brazil (miguel.chavez@ufmt.br)} \ \ \ \ \ Luis P. Yapu \thanks{Instituto de Matemática e Estatística, Universidade Federal Fluminense, Niterói - RJ, Brazil and 
Friedrich-Alexander-Universität Erlangen-Nürnberg, Germany
} \ \ \ \ \ \ Juan Límaco\thanks{Instituto de Matemática e Estatística, Universidade Federal Fluminense, Niterói - RJ, Brazil (jlimaco@id.uff.br).}}

\begin{document}
\maketitle

\numberwithin{equation}{section}

\textbf{Abstract:} In this article we establish the well-posedness, energy estimates, stability, and local null controllability for the thermistor system modeled by a parabolic-parabolic system using a
control force acting on just one equation of the system. The proof of the controllability is based on appropriate Carleman estimates and Liusternik's inverse function theorem to obtain the local controllability of the nonlinear system. The coupling of the system happens both in the terms of order zero and one, which requires the use of a special Carleman estimate for the system.
\\

\textbf{MSC Classification (2020)}: Primary: 35K65, 93B05; Secondary: 93C10. 

\textbf{keywords}: Thermistor model, Existence, Stability, Controllability, Nonlinear systems in Control Theory.

\section{Introduction}

Let $\Omega$ denote a bounded open, non-empty subset of $\mathbb{R}^N$ (with $N=1,2$ or $3$) with smooth boundary (at least $\mathcal{C}^3$). For any $T>0$, we denote $Q_T=\Omega\times (0,T)$ and $\Sigma_T=\partial \Omega \times (0,T)$. If $T=\infty$, we denote $Q_{\infty}=\Omega\times (0,\infty)$ and $\Sigma_{\infty} = \partial \Omega \times (0,\infty)$. 

We study a physical system describing a conductor with temperature $y(x,t)$ and electric potential $z(x,t)$. If $\mathcal{I}$ denotes the current density and $\mathcal{Q}$ the heat flow, then Ohm's and Fourier's laws are given, respectively, by
$$
\mathcal{I} = -\sigma(y) \nabla z, \qquad \mathcal{Q} = -\kappa(y) \nabla y,
$$
where $\kappa(y) > 0$ denotes the thermal conductivity and $\sigma(u)>0$ the electrical conductivity.

From conservation laws, we arrive to the following model that describes our physical system:
\begin{equation}
    \left\{
    \begin{array}{ll}
y_t - \nabla \cdot (\kappa(y) \nabla y) = \sigma(y) |\nabla z|^2, \\
z_t - \nabla \cdot (\sigma(y) \nabla z) = 0.
    \end{array}
    \right.
\end{equation}

Considering $0<T\leq \infty$, we study the following thermistor problem:
\begin{equation}\label{sistema_thermistor}
    \left\{
    \begin{array}{ll}
        y_t - \nabla \cdot (\kappa(y) \nabla y) = \sigma(y) |\nabla z|^2 + f, \quad &\text{in}\quad Q_{T}\\
        z_t - \nabla \cdot (\sigma(y) \nabla z) = g, \quad &\text{in}\quad Q_{T}, \\
        y=0, \quad z=z^*, \quad &\text{on}\quad \Sigma_{T}, \\
        y(x,0) = y_0(x), \ \  z(x,0)=z_0(x) \quad &\text{in}\quad \Omega,
    \end{array}
    \right.
\end{equation}

where $f$, $g$ are the forces acting inside the system \eqref{sistema_thermistor}, $z^*$ is the electric potential at the boundary of the system \eqref{sistema_thermistor}.\\
Hypothesis about the functions $f, g : Q_{\infty} \to \mathbb{R}$:
    \begin{hypothesis}\label{Hyp1}
        f, g \in \mathcal{X} := \Big\{ w \in L^2(Q_{\infty}) \ : \ w_t \in L^2(Q_{\infty}),\ w(0) \in H_0^1(\Omega) \Big\}.
    \end{hypothesis}
Hypothesis about the function $z^* : Q_{\infty} \to \mathbb{R}$:
    \begin{hypothesis}\label{Hyp2}
        z^* \in H^1(0,\infty;H^3(\Omega)),
    \end{hypothesis}
    \vspace{-1.3em}
    \begin{hypothesis}\label{Hyp3}
        \|z^*(t)\|_{H^3(\Omega)}^2 + \|z^*_t(t)\|_{H^3(\Omega)}^2 \leq C^* e^{-\gamma t},  \forall t > 0\  \text{ with } \ C^*,\gamma > 0.
    \end{hypothesis}
Hypothesis about the functions $\kappa, \sigma : \mathbb{R} \to \mathbb{R}$:
        \begin{hypothesis}\label{Hyp4}
            \kappa, \sigma \in C^2(\mathbb{R}),
        \end{hypothesis}
        \vspace{-1.3em}
        \begin{hypothesis}\label{Hyp5}
            0 < \kappa_0 \leq \kappa(\cdot) \leq \kappa_1, \ \  0 < \sigma_0 \leq \sigma(\cdot) \leq \sigma_1,
        \end{hypothesis}
        \vspace{-1em}
        \begin{hypothesis}\label{Hyp6}
            |\kappa'(\cdot)| + |\kappa''(\cdot)| + |\sigma'(\cdot)| + |\sigma''(\cdot)| \leq M\ \  \text{with} \   M>0.
        \end{hypothesis}
        
First, we reformulate the system \eqref{sistema_thermistor} in order to get homogeneous conditions on the boundary $\Sigma_T$ for both variables. Making $p=z-z^*$, we get the new equivalent system, 
\begin{equation}\label{sistema2_thermistor}
    \left\{
    \begin{array}{ll}
        y_t - \nabla \cdot (\kappa(y) \nabla y) = \sigma(y) |\nabla p|^2 + 2\sigma(y) \nabla p \cdot \nabla z^* + \sigma(y)|\nabla z^*|^2  +f,  &\text{in} \ Q_{T}\\
        p_t - \nabla \cdot (\sigma(y) \nabla p) - \nabla \cdot (\sigma(y) \nabla z^*) = - (z^*)_t + g,  &\text{in}\  Q_{T}, \\
        y=0, \quad p=0,  &\text{on}\  \Sigma_{T}, \\
        y(x,0) = y_0(x), \ \ \ p(x,0)= p_0(x) &\text{in}\  \Omega,
    \end{array}
    \right.
\end{equation}
where $p_0(x) = z_0(x) - z^*(x,0)$.\\

The main results of this paper are the following three theorems:
\begin{teo}\label{Theo1}
    Given $(y_{0}, p_0) \in [H^{3}(\Omega)\cap H^{1}_{0}(\Omega)]^2$ and assuming \eqref{Hyp1}-\eqref{Hyp2},  \eqref{Hyp4}-\eqref{Hyp6}. Then, the system \eqref{sistema2_thermistor} has a unique strong solution $(y,p)$ for small initial datum.
\end{teo}

\begin{teo}\label{Theo2}
    Given $(y_{0}, p_0) \in [H^{3}(\Omega)\cap H^{1}_{0}(\Omega)]^2$ and assuming \eqref{Hyp1}-\eqref{Hyp6}. Then, the system \eqref{sistema2_thermistor} is exponentially stable in $H^3$-norm for small initial datum.
\end{teo}

Now, let $\omega \subset \Omega$ be an open set of class $\mathcal{C}^1$ such that $\partial \Omega \cap \partial \omega \neq \emptyset$ and $T \in (0, \infty)$. We analyze the controllability of the thermistor problem \eqref{sistema2_thermistor} by an interior control $v$. That is, we study either
\begin{equation}\label{sistema2}
    \left\{
    \begin{array}{ll}
        y_t - \nabla \cdot (\kappa(y) \nabla y) = \sigma(y) |\nabla p|^2 + 2\sigma(y) \nabla p \cdot \nabla z^* + \sigma(y)|\nabla z^*|^2 + v 1_\omega,  &\text{in }\ Q_T\\
        p_t - \nabla \cdot (\sigma(y) \nabla p) - \nabla \cdot (\sigma(y) \nabla z^*) =  - (z^*)_t , &\text{in }\ Q_T, \\
        y=0, \quad p=0, &\text{on }\ \Sigma_T, \\
        y(x,0) = y_0(x), \ \ \ p(x,0)= p_0(x)  &\text{in }\ \Omega,
    \end{array}
    \right.
\end{equation}
or
\begin{equation}\label{sistema2_control2}
    \left\{
    \begin{array}{ll}
        y_t - \nabla \cdot (\kappa(y) \nabla y) = \sigma(y) |\nabla p|^2 + 2\sigma(y) \nabla p \cdot \nabla z^* + \sigma(y) |\nabla z^*|^2, \quad &\text{in}\quad Q_T\\
        p_t - \nabla \cdot (\sigma(y) \nabla p) - \nabla \cdot (\sigma(y) \nabla z^*) = - (z^*)_t + v 1_\omega , \quad &\text{in}\quad Q_T, \\
        y=0, \quad p=0, \quad &\text{on}\quad \Sigma_T, \\
        y(x,0) = y_0(x), \quad p(x,0) = p_0(x) \quad &\text{in}\quad \Omega,
    \end{array}
    \right.
\end{equation}
where $p_0(x) = z_0(x) - z^*(x,0)$.
\begin{defi}\label{def_1_Sec4}
    The system \eqref{sistema2} $($or \eqref{sistema2_control2}$)$ is said \emph{locally null controllable at time $T>0$}, if there exists $\epsilon>0$ such that, for every $(y_0, p_0) \in [H^3(\Omega) \cap H_0^1(\Omega)]^2$ and $z^*$ satisfying \eqref{Hyp2}, with  
    $$
    \|y_0\|_{H^3(\Omega)} + \|p_0\|_{H^3(\Omega)} + \|z^*\|_{H^{1}(0,\infty;H^{3}(\Omega))} < \epsilon,
    $$
    then there exists a control $v \in L^2(\omega \times (0,T))$ such that the solution of the system \eqref{sistema2} $($or \eqref{sistema2_control2}$)$ verifies $y(\cdot,T)=0$\ and\ $p(\cdot,T)=0$ in $\Omega$.
\end{defi}

Due to the change of variable $p=z-z^*$, the null controllability of \eqref{sistema2} (or \eqref{sistema2_control2}) is equivalent to the controllability of system \eqref{sistema2_thermistor}, in the following sense:
$y(T)=0$ and $z(T)=z^*(T)$ in $\Omega$.\\
If, additionally,
    \begin{hypothesis}\label{Hyp7}
        z^* \in W^{2,\infty}(Q_T)\ \ \text{and}\ \  \nabla z^* \in [C^3(\overline{Q}_T)]^N,
    \end{hypothesis}
    \vspace{-1.5em}
    \begin{hypothesis}\label{Hyp8}
        \centering
        \begin{split}
        \nabla z^*(x_0,t) \cdot \nu(x_0) \neq 0 \ \   \forall t \in [t_1, t_2] \ \   \text{ for some}\ x_0 \in\gamma \subset \partial\Omega \cap \partial \omega, \\
        \text{where } 0<t_1 < t_2 < T \text{ and } \  t_1 \text{ and } t_2 \text{ verify the condition } \eqref{Eq-Carleman-1}.
        \end{split}
    \end{hypothesis}
    Let $\rho=\rho(t)$ be a positive function that explodes exponentially at time  $t=T$,
    \begin{hypothesis}\label{Hyp9}
        \rho |\nabla z^*|^2, \rho |\nabla z^*_t|^2, \rho \Delta z^*, \rho \Delta z^*_t, \rho z^*_t, \rho z^*_{tt} \in L^2(Q_T),  
    \end{hypothesis}
    \vspace{-1em}
    \begin{hypothesis}\label{Hyp10}    
        |\nabla z^*(0)|^2, \sigma(0) \Delta z^*(0)- z^*_t(0) \in H^1_0(\Omega).
    \end{hypothesis}
\begin{teo}\label{maintheorem}
Assuming \eqref{Hyp4}-\eqref{Hyp10}, then for any $T\in (0,\infty)$ the systems \eqref{sistema2} and \eqref{sistema2_control2} are locally null controllable at time $T$.
\end{teo}

The thermistor problem and similar systems have been well studied in the literature. For the analogous parabolic-elliptic system where the second equation of \eqref{sistema_thermistor} is elliptic, existence, regularity, uniqueness and blow-up has been shown by Antontsev and Chipot \cite{Chipot-94} and Cimatti \cite{Cimatti-92, Cimatti-89}, who also studied the optimal control problem in \cite{Cimatti-07}.
The first result of null controllability at time $T$ of the parabolic-elliptic thermistor model was recently proved in \cite{Thermistor_HNLP-23}. 

For parabolic-parabolic systems, existence of weak solutions was proved by Rodrigues \cite{Rodrigues-92}. Afterwards, Lee and Shilkin \cite{Lee-Shilkin-05} studied the optimal control problem associated to a system simpler than \eqref{sistema_thermistor} where there is no quasi-linearity in the first equation. 
For our parabolic-parabolic system \eqref{sistema_thermistor}, as far as we know, our work is the first to address the null controllability at time $T$.
In \cite{FCS-_Guerrero_Puel-14} and \cite{EFC-Limaco-Menezes-PE-13}, the null controllability of a parabolic-elliptic system is achieved as the limit of the null controllability of parabolic systems, where the second equation of \eqref{sistema_thermistor} contains a term $\epsilon z_t$ for a small parameter $\epsilon>0$. In \cite{Thermistor_HNLP-23}, the null controllability of this parabolic system is posed as an open problem, which we have addressed in this work.

This paper is organized as follows. In Section \ref{sec2:well_posedness} we prove the well-posedness of our system \eqref{sistema2_thermistor} and give energy estimates. In Section \ref{sec3:stability} we prove the stability in $H^3$-norms of the system \eqref{sistema2_thermistor}. In Section \ref{sec4:controllability}  we calculate an appropriate Carleman estimate that allows us to prove the controllability of the linearized system of \eqref{sistema2_thermistor}. Furthermore, we prove a controllability result of the nonlinear system \eqref{sistema2_thermistor} using Liusternik's techniques. 
Finally in Section \ref{sec:final_remarks} we give a large-time controllability result that is proved as a consequence of the previous stability and controllability results and we discuss extensions and open problems.


\section{Existence of solution of the system \eqref{sistema2_thermistor} }
\label{sec2:well_posedness}

In this section we prove Theorem \ref{Theo1}. More precisely: with the hypotheses of Theorem \ref{Theo1}, for any $0< T \leq \infty$ there exists $r=r(\kappa_0, \kappa_1, \sigma_0, \sigma_1, M, \Omega, N)>0$ (independent of $T$), such that, if
$$ ||y_{0}||_{H^{3}(\Omega)} + ||p_{0}||_{H^{3}(\Omega)} + ||z^*||_{H^{1}(0,\infty;H^{3}(\Omega))} + ||f||_{\mathcal{X}}+||g||_{\mathcal{X}}\leq r,$$
the system \eqref{sistema2_thermistor} has a unique strong solution $(y,p)$  satisfying
\begin{equation}\label{A2}
    \begin{array}{l}
        ||y||_{H^{1}(0,T;H^{2}(\Omega))} + ||y_{tt}||_{L^{2}(Q_{T})} + ||p||_{H^{1}(0,T;H^{2}(\Omega))} + ||p_{tt}||_{L^{2}(Q_{T})}\\
        \noalign{\smallskip}\phantom{DDDDA}
        \displaystyle \leq C_0 C\left( ||y_{0}||_{H^{3}(\Omega)}, ||p_{0}||_{H^{3}(\Omega)} ,||z^*||_{H^{1}(0,\infty;H^{3}(\Omega))}, ||f||_{\mathcal{X}}, ||g||_{\mathcal{X}}\right),
    \end{array}
\end{equation}
where $C_0$ is a positive constant only depending  on $\Omega, \sigma_0,\, \kappa_0,\, \sigma_{1}$, $\kappa_1$ and $M$.

In fact, we employ the Faedo-Galerkin method with  the Hilbert basis of $H^{1}_{0}(\Omega)$ given by the eigenvectors $(w_{j})$ of the spectral problem $\left((w_{j},v)\right)=\lambda_{j} (w_{j},v),$ for all $v\in V=H^{2}(\Omega)\cap H^{1}_{0}(\Omega)$ and $j=1,2,3,\dots$ 

We denote by $V_{m}$ the subspace of $V$ generated by the vectors $\{w_1,w_2,...,w_m\}$ and we propose the following approximate problem:
\begin{equation}\label{A3}
\left\{
\begin{array}{l}
\left(y'_{m},v\right)+\left(\kappa( y_{m})\nabla y_{m},\nabla v\right) - \Big(\sigma(y_{m})(|\nabla p_{m}|^{2}+|\nabla z^*|^{2} + 2\nabla z^* \cdot\nabla p_m ), v \Big) \\
    \noalign{\smallskip} 
    \phantom{AAAAAAAAAAAAAAAAAAAAAAAAAAA}= \left(f,v\right),\,\, \forall \, v\in V_{m}, \vspace{0.1cm}\\
	\noalign{\smallskip} 
	(p_m',w)+\left(\sigma(y_{m})\nabla p_m,\nabla w\right)+\left(\sigma(y_m)\nabla z^*, \nabla w\right)=\left(g,w\right),\,\, \forall\, w\in V_{m}, \vspace{0.1cm}\\
		\noalign{\smallskip} 
y_{m}(0)=y_{0m} \to y_{0}, \ \ p_m(0) = p_{0m} \to p_0  \ \ \ \mbox{in}\,\, H^{3}(\Omega)\cap H^{1}_{0}(\Omega).
\end{array}
\right.
\end{equation}

The existence and uniqueness of the solution to \eqref{A3} in $(0,T_m)$ are ensured by classical ODE theory. The following estimates show that, in fact, they are defined for all $t\in (0, T)$.

In the sequel, the symbol $\tilde{C}_{i}$ is a constant that depends only on $\sigma_0,\,\kappa_0, \sigma_1,\, \kappa_1,\, M$ and $\Omega$ for $i \in \mathbb{N}$ with $1 \leq i \leq 18$. Moreover, here and in the next sections, the norm $\| \cdot \|$ will denote the $L^2$-norm $\| \cdot \|_{L^2(\Omega)}$.

{\bf Estimate I:} Taking $v=y_{m}(t)$ and $w=p_m(t)$ in \eqref{A3}, we deduce that
\begin{equation}\label{A-eq1}
\begin{array}{l}
\displaystyle \frac{1}{2}\frac{d}{dt}||y_m(t)||^{2}+\frac{\kappa_0}{2}|| \nabla y_m(t)||^{2}\leq \displaystyle  \frac{\kappa_{0}}{4}||\nabla y_{m}(t)||^{2} + \frac{\kappa_{0}}{32}||\Delta y_{m}(t)||^{2} + \frac{\sigma_{0}}{4}||\nabla p_{m}(t)||^{2}\\
\noalign{\smallskip}\phantom{SD\frac{1}{2}\frac{d}{dt}||y_m(t)||^{2}+\frac{\kappa_0}{2}|| \nabla y_m(t)||^{2}}

+\tilde{C}_{1}\left(  ||\Delta y_{m}(t)||^{4} + ||z^*(t)||^{4}_{H^{1}(\Omega)} \right) ||\nabla p_{m}(t)||^{2} \\
\noalign{\smallskip}\phantom{SD\frac{1}{2}\frac{d}{dt}||y_m(t)||^{2}+\frac{\kappa_0}{2}|| \nabla y_m(t)||^{2}}

+\tilde{C}_{1}\ ||z^*(t)||^{4}_{H^{1}(\Omega)} ||\Delta y_{m}(t)||^{2}\\
\noalign{\smallskip}\phantom{SD\frac{1}{2}\frac{d}{dt}||y_m(t)||^{2}+\frac{\kappa_0}{2}|| \nabla y_m(t)||^{2}}

+ \tilde{C}_1 \left( ||f(t)||^{2} + ||  z^*(t)||_{H^1(\Omega)}^2\right),
\end{array}
\end{equation}
\begin{equation}\label{A-eq2}
    \displaystyle \frac{1}{2}\frac{d}{dt}||p_m(t)||^{2} + \frac{\sigma_0}{2}||\nabla p_{m}(t)||^{2}\leq \tilde{C}_{2}\left(||g(t)||^{2} + ||z^*(t)||^{2}_{H^{1}(\Omega)}\right).
\end{equation}

From \eqref{A-eq1} and \eqref{A-eq2}, one has
\begin{equation}\label{A-eq2-2}
\begin{array}{l}
\displaystyle \frac{1}{2}\frac{d}{dt} \Big(||y_{m}(t)||^{2} + ||p_{m}(t)||^{2} \Big) + \frac{\kappa_{0}}{4}||\nabla y_{m}(t)||^{2}\\

\noalign{\smallskip}\phantom{SSSSSSS}

+ \displaystyle \left[\frac{\sigma_{0}}{4}-\tilde{C}_{1}\left(||\Delta y_{m}(t)||^{4}+||z^*(t)||^{4}_{H^{1}(\Omega)}\right)\right]||\nabla p_{m}(t)||^{2}\\

\noalign{\smallskip}\phantom{SSSSSSS}

\displaystyle \leq  \frac{\kappa_{0}}{32}||\Delta y_{m}(t)||^{2} + \tilde{C}_1 || z^*(t) ||_{H^1(\Omega)}^4 || \Delta y_m(t) ||^2\\

\phantom{SSSSSSSS}+ \tilde{C}_{1} \left(||f(t)||^{2} + ||z^*(t)||^{2}_{H^{1}(\Omega)} \right) + \tilde{C}_{2}\left(||g(t)||^{2} + ||z^*(t)||^{2}_{H^{1}(\Omega)}\right).

\end{array}
\end{equation}

{\bf Estimate II:} Taking $v=- \Delta y_{m}(t)$ and $w=-\Delta p_{m}(t)$ in \eqref{A3}, we see that
\begin{equation}\label{A-eq4}
    \begin{array}{l}
        \displaystyle \frac{1}{2}\frac{d}{dt}||\nabla y_{m}(t)||^{2}+\left[\frac{\kappa_{0}}{4}-\tilde{C}_{3} \Big(||\Delta y_{m}(t)||^4 + ||\Delta p_{m}(t)||^4 \Big) \right]||\Delta y_{m}(t)||^{2}\\
        \noalign{\smallskip}\phantom{\frac{\sigma_0}{2}||\Delta p_m(t)||^{2}T}

        \displaystyle \leq \frac{\sigma_{0}}{32}||\Delta p_{m}(t)||^{2} + \frac{\kappa_{0}}{32}||\Delta y_{m}(t)||^{2}\\ 
        \noalign{\smallskip}\phantom{\frac{\sigma_0}{2}||\Delta p_m(t)||^{2}TA}
        
        + \tilde{C}_{3}\ ||z^*(t)||^{4}_{H^{2}(\Omega)}||\Delta p_{m}(t)||^{2}\\

        \noalign{\smallskip} \phantom{\frac{\sigma_0}{2}||\Delta p_m(t)||^{2}TA}
        \displaystyle + \tilde{C}_{3}\left(||z^*(t)||^{4}_{H^{2}(\Omega)}+||f(t)||^{2} \right).
    \end{array}
\end{equation}
\begin{equation}\label{A-eq3}
    \begin{array}{l}
        \displaystyle \frac{1}{2}\frac{d}{dt}||\nabla p_{m}(t)||^{2} + \left(\frac{\sigma_{0}}{4} - \tilde{C}_{4}||\Delta y_{m}(t)||^4 \right)||\Delta p_{m}(t)||^{2}\\
        \noalign{\smallskip}\phantom{\frac{\sigma_0}{2}||\Delta \varphi_m||^{2}T}

        \displaystyle \leq \frac{\kappa_0}{32}||\Delta y_m(t)||^{2} + \frac{\sigma_0}{32}||\Delta p_m(t)||^{2} \\
        \noalign{\smallskip}\phantom{\frac{\sigma_0}{2}||\Delta \varphi_m||^{2}T}
        
        + \tilde{C}_{4}\ ||z^*(t)||^{4}_{H^{2}(\Omega)} ||\Delta y_{m}(t)||^{2} \\
        \noalign{\smallskip}\phantom{\frac{\sigma_0}{2}||\Delta \varphi_m||^{2}T}
        
        + \tilde{C}_{4}\left(||g(t)||^{2} + ||z^*(t)||^{2}_{H^{2}(\Omega)}\right)
    \end{array}
\end{equation}

{\bf Estimate III:} Taking $v=- \Delta y'_m(t)$ and $w=- \Delta p'_m(t)$ in \eqref{A3}, we have
\begin{equation}\label{A-eq5}
    \begin{array}{l}
        \displaystyle \frac{1}{2}\frac{d}{dt}\left(\int_{\Omega}\kappa( y_m)|\Delta y_m (x,t)|^{2}\, dx \right)+\frac{1}{2}||\nabla y'_{m}(t)||^{2}\\
	   \noalign{\smallskip} \phantom{\frac{1}{2}||\nabla y'_{m}||DDD}
	
        \displaystyle \leq \frac{\kappa_{0}}{32}||\Delta y_{m}(t)||^{2}  + \frac{\kappa_{0}}{32}||\Delta y'_{m}(t)||^{2}  + \frac{\sigma_{0}}{32}||\Delta p_{m}(t)||^{2}\\
	   \noalign{\smallskip} \phantom{\frac{1}{2}||\nabla y'_{m}||DDDD}

        \displaystyle +\tilde{C}_{5}\ ||\Delta y_m(t)||^{4} ||\Delta y_m(t)||^{2}\\
	   \noalign{\smallskip} \phantom{\frac{1}{2}||\nabla y'_{m}||DDDD}

        \displaystyle+ \tilde{C}_{5}\left( ||z^*(t)||^{4}_{H^{2}(\Omega)} + ||\Delta p_m(t)||^{4}\right)||\Delta y'_m(t)||^{2}\\
        \noalign{\smallskip} \phantom{\frac{1}{2}||\nabla y'_{m}||DDDD}

        +\tilde{C}_{5}\left(||f(t)||^{2}+||z^*(t)||^{2}_{H^{2}(\Omega)}\right).
    \end{array} 
\end{equation}
\begin{equation}\label{A-eq5_0}
\begin{array}{l}
\displaystyle \frac{1}{2}\frac{d}{dt}\left(\int_{\Omega}\sigma( p_m)|\Delta p_m (x,t)|^{2}\, dx \right) + \frac{1}{2}||\nabla p'_{m}(t)||^{2}\\
	\noalign{\smallskip} \phantom{\frac{1}{2}||\nabla p'_{m}||DDD}
	
\displaystyle \leq \frac{\kappa_{0}}{32}||\Delta y'_{m}(t)||^{2}  + \frac{\sigma_{0}}{32}||\Delta p_{m}(t)||^{2} + \frac{\sigma_{0}}{32}||\Delta p'_{m}(t)||^{2} \\
	\noalign{\smallskip} \phantom{\frac{1}{2}||\nabla y'_{m}||DDDD}

\displaystyle+ \tilde{C}_{6}\ ||\Delta y_m(t)||^{4} ||\Delta p'_m(t)||^{2}\\
	\noalign{\smallskip} \phantom{\frac{1}{2}||\nabla y'_{m}||DDDD}
    
\displaystyle +\tilde{C}_{6}\left(||g(t)||^{2}+||z^*(t)||^{2}_{H^{2}(\Omega)}\right).

\end{array} 
\end{equation}

{\bf Estimate IV:} Taking derivative with respect to $t$ in equations $\eqref{A3}_{1}$-$\eqref{A3}_{2}$ and then putting $v=-\Delta y'_{m}(t)$, $w=-\Delta p'_m$,  $v=y''_{m}$ and  $w=p''_{m}$, we have
\begin{equation}\label{A-eq6}
    \begin{array}{l}
        \displaystyle \frac{1}{2}\frac{d}{dt}||\nabla y'_m(t)||^{2}+\left[\frac{\kappa_0}{4} -\tilde{C}_{7}\left(||\Delta y_{m}(t)||^{4} + ||\Delta p_{m}(t)||^{4} + ||z^*(t)||^{4}_{H^{2}(\Omega)}\right)\right]|| \Delta y'_m(t)||^{2}\\
	  \noalign{\smallskip} \phantom{\frac{1}{2}||\nabla y'_{m}(t)||AA}
	
        \displaystyle \leq \frac{\kappa_{0}}{32}||\Delta y'_{m}(t)||^{2} + \frac{\sigma_{0}}{32}||\Delta p'_{m}(t)||^{2} \\
        \noalign{\smallskip} \phantom{\frac{1}{2}||\nabla y'_{m}(t)||AA}
        
        + \tilde{C}_{7}\Big(||z^*_t(t)||_{H^2(\Omega)}^{2}+||f_t(t)||^{2}\Big),
    \end{array}
\end{equation}
\begin{equation}\label{A-eq7}
    \begin{array}{l}
        \displaystyle \frac{1}{2}\frac{d}{dt}||\nabla p'_m(t)||^{2} +\left(\frac{\sigma_0}{4}-\tilde{C}_{8}||\Delta y_m(t)||^{4}\right)||\Delta p'_m(t)||^{2}\\
	   \noalign{\smallskip} \phantom{\frac{1}{2}||\nabla y'_{m}(t)||AA}
	
         \displaystyle \leq \frac{\kappa_{0}}{32}||\Delta y'_{m}(t)||^{2} + \frac{\sigma_{0}}{32}||\Delta p'_{m}(t)||^{2} \\
        \noalign{\smallskip} \phantom{\frac{1}{2}||\nabla y'_{m}(t)||AA}
        
        \displaystyle + \tilde{C}_{8}\left(||\Delta p_{m}(t)||^{4} + ||z^*(t)||^{4}_{H^{2}(\Omega)}\right)||\Delta y'_{m}(t)||^{2}\\
	   \noalign{\smallskip} \phantom{\frac{1}{2}||\nabla y'_{m}(t)||AAA}
 
        \displaystyle + \tilde{C}_{8}\left(||z^*_t(t)||^{2}_{H^{2}(\Omega)}+||g_t(t)||^{2}\right).
    \end{array}
\end{equation}
Also,
\begin{equation}\label{A-eq7-2}
    \begin{array}{l}
        \displaystyle \frac{1}{2}||y''_{m}(t)||^{2} +  \frac{1}{2}\frac{d}{dt} \left( \int_\Omega \kappa(y_m) |\nabla y'_m(x,t)|^2 dx \right)  \leq \frac{\kappa_{0}}{32}||\Delta y'_{m}(t)||^{2} \\
        \noalign{\smallskip} \phantom{\frac{1}{2}||\nabla y'_{m}(t)||S}
        
        + \displaystyle \frac{\sigma_{0}}{32}||\Delta p_{m}(t)||^{2} + \frac{\sigma_{0}}{32}||\Delta p'_{m}(t)||^{2}\\
        \noalign{\smallskip} \phantom{\frac{1}{2}||\nabla y'_{m}(t)||S}

        + \tilde{C}_9\ || z^*_t(t) ||_{H^2(\Omega)}^4 || \Delta p_m(t) ||^2\\ 
        \noalign{\smallskip} \phantom{\frac{1}{2}||\nabla y'_{m}(t)||S}	

        + \tilde{C}_{9} \left( ||\Delta p_{m}(t)||^{4} + || z^*(t) ||_{H^2(\Omega)}^4  \right) ||\Delta p'_{m}(t)||^{2} \\
        \noalign{\smallskip} \phantom{\frac{1}{2}||\nabla y'_{m}(t)||S}	

        \displaystyle +\tilde{C}_{9}\left(||\nabla y'_{m}(t)||^4 + ||\Delta y_{m}(t)||^{4} + ||\Delta p_{m}(t)||^{4} + ||z^*(t)||^{4}_{H^{2}(\Omega)}\right)||\Delta y'_{m}(t)||^{2}\\

        \noalign{\smallskip} \phantom{\frac{1}{2}||\nabla y'_{m}(t)||S}	

        \displaystyle +\tilde{C}_{9}\left(||f_t(t)||^{2}+||z^*(t)||^{2}_{H^{2}(\Omega)}||z^*_t(t)||^{2}_{H^{2}(\Omega)}\right)
    \end{array}
\end{equation}
\begin{equation}\label{A-eq7-3}
    \begin{array}{l}
        \displaystyle  \frac{1}{2} || p''_m(t) ||^2 +   \frac{1}{2}\frac{d}{dt}\left(\int_{\Omega}\sigma(y_{m}) |\nabla p'_{m}(x,t)|^{2}\,dx\right) \leq \frac{\sigma_{0}}{32}||\Delta p'_{m}(t)||^{2} \\
        \noalign{\smallskip} \phantom{A}	

        \displaystyle + \frac{\kappa_{0}}{32}||\Delta y_{m}(t)||^{2} + \frac{\kappa_{0}}{32}||\Delta y'_{m}(t)||^{2} \\
        \noalign{\smallskip} \phantom{A}	

        \displaystyle +\tilde{C}_{10}\left( || \Delta y_{m}(t)||^{4} + ||\Delta p_{m}(t)||^{4} + ||\nabla p'_{m}(t)||^4 \right.\\
        \noalign{\smallskip} \phantom{AAAA}

        \left.+ ||z^*(t)||^{4}_{H^{2}(\Omega)} +  ||z^*_t(t)||^{4}_{H^{2}(\Omega)}   \right)|| \Delta y'_{m}(t)||^{2}\\
        \noalign{\smallskip} \phantom{A}	

        \displaystyle + \tilde{C}_{10} \left( ||z^*_t(t)||^{2}_{H^{2}(\Omega)}  + ||g_t(t)||^{2} \right).
    \end{array}
\end{equation}
From all these estimates, we have
\begin{equation}\label{A-eq8}
    \begin{array}{l}
        \displaystyle \frac{1}{2}\frac{d}{dt}\Big\{||y_m(t)||^{2}  + ||\nabla y_m(t)||^{2}  + ||\nabla y'_m(t)||^{2} \\ 
        \noalign{\smallskip}\phantom{QQ}

         + ||p_m(t)||^{2} + ||\nabla p_m(t)||^{2} + ||\nabla p'_m(t)||^{2}  \\ 
        \noalign{\smallskip}\phantom{QQ}

        + \displaystyle  \int_{\Omega} \kappa(y_m)|\Delta y_m(x,t)|^{2}\,dx + \int_{\Omega} \kappa(y_{m}) |\nabla y'_{m}(x,t)|^{2}\,dx \\
        \noalign{\smallskip}\phantom{QQ}

        + \displaystyle \left. \int_{\Omega} \sigma(y_m)|\Delta p_m(x,t)|^{2}\,dx + \int_{\Omega}\sigma(y_{m}) |\nabla p'_{m}(x,t)|^{2}\,dx \right\}\\
        \noalign{\smallskip}\phantom{QQ}

        \displaystyle +\frac{\kappa_{0}}{4}||\nabla y_m(t)||^{2}  + \frac{1}{2} ||y''_{m}(t)||^{2}  + \frac{1}{2}||\nabla y'_{m}(t)||^{2} +\frac{1}{2} ||p''_{m}(t)||^{2}+ \frac{1}{2}||\nabla p'_{m}(t)||^{2}\\
        \noalign{\smallskip}\phantom{QQ}
        
        +\displaystyle \left[\frac{\sigma_{0}}{8}-\tilde{C}_{1}\left(||\Delta y_{m}(t)||^{4}+||z^*(t)||^{4}_{H^{2}(\Omega)} \right) \right]||\nabla p_{m}(t)||^{2}\\
        \noalign{\smallskip}\phantom{QQ}

        \displaystyle +\left[\frac{\kappa_0}{8}-\tilde{C}_{11}\left(||\Delta y_m(t)||^{4} + ||\Delta p_m(t)||^{4} + ||z^*(t)||^{4}_{H^{2}(\Omega)} \right)\right] ||\Delta y_m(t)||^{2}\\
        \noalign{\smallskip}\phantom{QQ}

        +\displaystyle \left[\frac{\sigma_{0}}{8}-\tilde{C}_{12}\left(||\Delta y_{m}(t)||^{4} + ||z^*(t)||^{4}_{H^{2}(\Omega)}  + ||z_t^*(t)||^{4}_{H^{2}(\Omega)} \right) \right] ||\Delta p_{m}(t)||^{2}\\
        \noalign{\smallskip}\phantom{QQ}

        \displaystyle +\left[\frac{\kappa_0}{8}-\tilde{C}_{13} \left(|| \Delta y_m(t)||^{4} + ||\nabla y'_{m}(t)||^4 + ||\Delta p_{m}(t)||^{4} + ||\nabla p'_{m}(t)||^4 \right. \right.\\
        \noalign{\smallskip}\phantom{QQQQQQQQQ}
        
        \displaystyle \left. \left. + ||z^*(t)||^{4}_{H^{2}(\Omega)} + ||z^*_t(t)||^{4} \right) \right] ||\Delta y'_m(t)||^{2}\\
        \noalign{\smallskip}\phantom{QQ}

        \displaystyle +\left[\frac{\sigma_0}{8}-\tilde{C}_{14}\left(|| \Delta y_m(t)||^{4} + ||\Delta p_{m}(t)||^{4} + ||z^*(t)||^{4}_{H^{2}(\Omega)} \right)\right] ||\Delta p'_m(t)||^{2}\\
        \noalign{\smallskip}\phantom{QDD}

        \displaystyle \leq \tilde{C}_{15} \Big( ||f(t)||^{2} + ||f_t(t)||^{2}+||g(t)||^{2} + ||g_t(t)||^{2} \\
        \noalign{\smallskip}\phantom{QDDTdddD}

        \displaystyle\left. + ||z^*(t)||_{H^2(\Omega)}^{2} + ||z^*_t(t)||^{2}_{H^2(\Omega)} + ||z^*(t)||^{4}_{H^2(\Omega)} + ||z^*_t(t)||^{4}_{H^{2}(\Omega)} \right),
    \end{array}
\end{equation}
where  $\displaystyle \tilde{C}_{11}=\tilde{C}_1 + \tilde{C}_{3}+\tilde{C}_{4}+\tilde{C}_{5}$,\   $\displaystyle \tilde{C}_{12}=\tilde{C}_{3}+\tilde{C}_{4}+\tilde{C}_{9}$, \  $\displaystyle \tilde{C}_{13}=\tilde{C}_{5} + \tilde{C}_{7} + \tilde{C}_{8} + \tilde{C}_{9} +\tilde{C}_{10}$, \ \ $\displaystyle \tilde{C}_{14}=\tilde{C}_{6}+\tilde{C}_{8}+\tilde{C}_{9}$\ and   \ $\displaystyle \tilde{C}_{15}=\sum_{i=1}^{10}\tilde{C}_{i}$.

Also, from \eqref{A3} taking $v=-\Delta y'_m$ and $w = -\Delta p'_m$. We have
\begin{align*}
    ||\nabla y'_m(0)||^{2} &\leq \tilde{C}_{16} \left(||y_0||^{2}_{H^{3}(\Omega)} +||y_0||^{4}_{H^{2}(\Omega)} + ||y_0||^{6}_{H^{2}(\Omega)} + ||p_0||^{4}_{H^{3}(\Omega)} + ||p_0||^{8}_{H^{1}_0(\Omega)} \right.\\
                            &\ \ \ \ \ \ \ \ \ + \left.  ||z^*||^{4}_{L^{\infty}(0,\infty;H^{3}(\Omega))} + ||z^*||^{8}_{L^{\infty}(0,\infty;H^{1}_0(\Omega))}  +||f(0)||^{2}_{H^{1}_{0}(\Omega)}\right)
\end{align*}
\begin{align*}
    ||\nabla p_{m}'(0)||^{2} &\leq \tilde{C}_{17} \left(||y_0||^{4}_{H^{3}(\Omega)} + ||y_0||^{8}_{H^{2}(\Omega)} + ||p_0||^{2}_{H^{3}(\Omega)} + ||p_0||^{4}_{H^{2}(\Omega)} \right.\\
                            &\ \ \ \ \ \ \ \ \ +  \left.  ||z^*||^{2}_{L^{\infty}(0,\infty;H^{3}(\Omega))}  + ||z^*||^{4}_{L^{\infty}(0,\infty;H^{2}(\Omega))} +||g(0)||^{2}_{H^{1}_{0}(\Omega)}\right).
\end{align*}

There exists $r = r(\kappa_0, \kappa_1, \sigma_0, \sigma_1, M, \Omega, N)>0$ such that for 
$$||y_{0}||_{H^{3}(\Omega)} + ||p_{0}||_{H^{3}(\Omega)} + ||f||_{\mathcal{X}} + ||g||_{\mathcal{X}} + ||z^*||_{H^{1}(0,\infty;H^{3}(\Omega))}< r,$$
we have
\begin{equation}\label{A-eq9}
\begin{array}{l}
\displaystyle L_{0}+ L_1 < \frac{\gamma_0}{2},
\end{array}
\end{equation}
where 
$$
\left\{\begin{array}{l}
\displaystyle \gamma_0=\min \left\{\frac{\kappa_0}{8},\frac{\sigma_0}{8}\right\}, \vspace{0.1cm}\\

\displaystyle L_{0} = \left( \tilde{C}_{1} + \tilde{C}_{11} + \tilde{C}_{12} + \tilde{C}_{13} + \tilde{C}_{14} \right) ||\Delta y_{0}||^{4} + \left( \tilde{C}_{11} + \tilde{C}_{13} + \tilde{C}_{14} \right) ||\Delta p_{0}||^{4},\vspace{.1cm}\\

\displaystyle L_{1} = \left(\tilde{C}_{1} + \tilde{C}_{11} + \tilde{C}_{13} + \tilde{C}_{14} \right) ||z^*||^{4}_{L^{\infty}(0,\infty;H^{2}(\Omega))} + \left(\tilde{C}_{12} + \tilde{C}_{13}\right) ||z^*_t||^{4}_{L^{\infty}(0,\infty;H^{2}(\Omega))},
\end{array}\right.
$$
and 
\begin{equation}\label{A-eq10}
\left\{\begin{array}{l}

\tilde{C}_{15} \left( ||f||^{2}_{H^{1}(0,\infty;L^{2}(\Omega))} + ||g||^{2}_{H^{1}(0,\infty;L^{2}(\Omega))} + || z^*||^{2}_{H^{1}(0,\infty;H^{2}(\Omega))} + ||z^*||^{4}_{H^{1}(0,\infty;H^{2}(\Omega))} \right)\\
\noalign{\smallskip}

+\displaystyle\frac{(\kappa_1 + 1)}{2}\ \tilde{C}_{16} \left(||y_0||^{2}_{H^{3}(\Omega)} +||y_0||^{4}_{H^{2}(\Omega)} + ||y_0||^{6}_{H^{2}(\Omega)} + ||p_0||^{4}_{H^{3}(\Omega)} + ||p_0||^{8}_{H^{1}_0(\Omega)} \right.\\
 \noalign{\smallskip}                          
                            
\ \ \ \ \ \ \ \ \ \ \ \ \ \ \ \ \ \ + \left.  ||z^*||^{4}_{L^{\infty}(0,\infty;H^{3}(\Omega))} + ||z^*||^{8}_{L^{\infty}(0,\infty;H^{1}_0(\Omega))}  +||f(0)||^{2}_{H^{1}_{0}(\Omega)}\right)\\
\noalign{\smallskip}

+\displaystyle\frac{(\sigma_1+1 )}{2}\ \tilde{C}_{17} \left(||y_0||^{4}_{H^{3}(\Omega)} + ||y_0||^{8}_{H^{2}(\Omega)} + ||p_0||^{2}_{H^{3}(\Omega)} + ||p_0||^{4}_{H^{2}(\Omega)} \right.\\
\noalign{\smallskip}

\ \ \ \ \ \ \ \ \ \ \ \ \ \ \ \ \ \ +  \left.  ||z^*||^{2}_{L^{\infty}(0,\infty;H^{3}(\Omega))}  + ||z^*||^{4}_{L^{\infty}(0,\infty;H^{2}(\Omega))} +||g(0)||^{2}_{H^{1}_{0}(\Omega)}\right)\\
\noalign{\smallskip}

\displaystyle+ \frac{1}{2}\left(||y_0||^{2}+||\nabla y_0||^{2}+\kappa_{1}||\Delta y_0||^{2} + ||p_0||^{2}+||\nabla p_0||^{2}+\kappa_{1}||\Delta p_0||^{2} \right)\\

\noalign{\smallskip}\phantom{T}
\displaystyle 
< \left(2 \gamma_0 + \frac{1}{4} \right) r^2.
\end{array}\right.
\end{equation}

Using these inequalities, we can prove by a contradiction argument that, for every $t\in (0,T_m)$,  the following estimate holds:
$$
\left(\tilde{C}_1 +\tilde{C}_{11}+\tilde{C}_{12}+\tilde{C}_{13}+\tilde{C}_{14}\right) ||\Delta y_{m}(t)||^{4} + \left(\tilde{C}_{11}+\tilde{C}_{13}+\tilde{C}_{14}\right) ||\Delta p_{m}(t)||^{4}< \frac{\gamma_0}{2}.
$$
Taking $t\in (0,T_m)$, we can integrate in time variable from $0$ to $t$:
\begin{equation}\label{A-eq11}
\begin{array}{l}
\displaystyle \frac{1}{2} \Big(||y_m(t)||^{2}+||\nabla y_m(t)||^{2} + (1+ \kappa_0)||\nabla y'_{m}(t)||^{2} + \kappa_{0}|| \Delta y_m(t)||^{2} \\
\noalign{\smallskip}\phantom{Q}

+ ||p_m(t)||^{2}+||\nabla p_m(t)||^{2} + (1+\sigma_0)||\nabla p'_{m}(t)||^{2} + \sigma_{0}||\Delta p_{m}(t)||^{2}\Big)\\
\noalign{\smallskip}\phantom{QQ}

+\displaystyle \frac{\kappa_{0}}{4} \int_0^t ||\nabla y_m(s)||^{2} ds + \frac{1}{2} \int_0^t ||y''_{m}(s)||^{2}ds  + \frac{1}{2} \int_0^t ||\nabla y'_{m}(s)||^{2} ds\\ 
\noalign{\smallskip}\phantom{QQ}

+\displaystyle \frac{\kappa_{0}}{16} \int_0^t ||\Delta y_m(s)||^{2} ds + \frac{\kappa_{0}}{16} \int_0^t ||\Delta y'_m(s)||^{2} ds\\
\noalign{\smallskip}\phantom{QQ}

+\displaystyle \frac{\sigma_{0}}{16} \int_0^t ||\nabla p_m(s)||^{2} ds + \frac{1}{2} \int_0^t ||p''_{m}(s)||^{2} ds + \frac{1}{2} \int_0^t ||\nabla p'_{m}(s)||^{2} ds\\
\noalign{\smallskip}\phantom{QQ}

+\displaystyle \frac{\sigma_{0}}{16} \int_0^t ||\Delta p_m(s)||^{2} ds + \frac{\sigma_{0}}{16} \int_0^t ||\Delta p'_m(s)||^{2} ds\\
\noalign{\smallskip}\phantom{QQ}

\leq \tilde{C}_{18}\left(||y_0||^{4}_{H^{3}(\Omega)} + ||y_0||^{2}_{H^{3}(\Omega)} + ||y_0||^{8}_{H^{2}(\Omega)} + ||y_0||^{6}_{H^{2}(\Omega)}  \right.\\
\noalign{\smallskip}\phantom{QhhWWA}

+||p_0||^{4}_{H^{3}(\Omega)} + ||p_0||^{2}_{H^{3}(\Omega)} + ||p_0||^{8}_{H^{1}_0(\Omega)} +||f||^{2}_{\mathcal{X}}+||g||^{2}_{\mathcal{X}} \\
\noalign{\smallskip}\phantom{QhhWWA}

\displaystyle + ||z^*||^{4}_{L^{\infty}(0,\infty;H^{3}(\Omega))}  + ||z^*||^{8}_{L^{\infty}(0,\infty;H^{1}_0(\Omega))}  \\
\noalign{\smallskip}\phantom{QhhWWA}

\displaystyle \left.+ ||z^*||^{2}_{H^{1}(0,\infty;H^{2}(\Omega))} + ||z^*||^{4}_{H^{1}(0,\infty;H^{2}(\Omega))}\right).
\end{array}
\end{equation}

Notice that the term of the right side in \eqref{A-eq11} do not dependent on $m$. Thus, we can extend the solution $(y_m, p_m)$ to the interval $(0,T)$ and proceeding similarly we obtain the estimate \eqref{A-eq11} for $t\in (0,T).$

Now, we obtain from estimates \eqref{A-eq8} and \eqref{A-eq11} that
$$
\left\{\begin{array}{l}
(y_m), (p_m)\,\, \mbox{ are bounded in }\,\, L^{\infty}(0,T;H^{2}(\Omega))\cap L^{2}(0,T;H^{1}_{0}
(\Omega)),\\
\noalign{\smallskip} 
(y'_{m}), (p'_m)\,\, \mbox{ are bounded in }\,\, L^{\infty}(0,T;H^{1}_{0}(\Omega))\cap L^{2}(0,T;H^{2}(\Omega)),\\
\noalign{\smallskip} 
(y''_{m}), (p''_{m}) \,\, \mbox{ are bounded in }\,\, L^{2}(Q_{T}).
\end{array}\right.
$$
All these uniform bounds allow us to take limits in \eqref{A3} (at least for a subsequence) as $m\to \infty$. Indeed, the unique delicate point is the a.e convergence of $\sigma(y_m)$ and $\kappa(y_m)$. But this is a consequence of the fact that the sequence $(y_m)$ is pre-compact in $L^{2}(0,T;H^{1}(\Omega))$ and $\sigma,\,\,\kappa \in C^{2}(\mathbb{R})$.

The uniqueness of the strong solution to \eqref{sistema2_thermistor} can be proved by argument standards (see \cite{Rincon}).

Due to \eqref{A-eq8}, we have that \eqref{A2} holds.
This ends the proof.



\section{Stability of the system \eqref{sistema2_thermistor}}
\label{sec3:stability}

Using the results of Section \ref{sec2:well_posedness}, we obtain Theorem \ref{Theo2} (stability). More precisely, we prove that there exists $r > 0$ such that, if $z^*$ satisfies \eqref{Hyp2}-\eqref{Hyp3},  and
    $$
    \|y_0\|_{H^3(\Omega)} + \|p_0\|_{H^3(\Omega)} + \|z^*\|_{H^1(0,\infty; H^3(\Omega))} < r,
    $$
then the solution $(y,p)$ of system \eqref{sistema2_thermistor} with $f=g=0$ verifies that there exist positive constants $\rho$ and $\tilde{C}$  such that
    $$
    \|y(t)\|_{H^3(\Omega)}^2 + \|p(t)\|_{H^3(\Omega)}^2 \leq \tilde{C} e^{-\rho t}, 
    $$
where $\tilde{C} = \tilde{C}(\kappa_0, \kappa_1, \sigma_0, \sigma_1, M,\Omega, N, \rho, r)$.

\begin{proof} (of Theorem \ref{Theo2}) 
Taking $f=g=0$ in Theorem \ref{Theo1} we get a constant $r>0$ such that if
$$
\|y_0\|_{H^3(\Omega)} + \|p_0\|_{H^3(\Omega)} + \|z^*\|_{H^1(0,\infty; H^3(\Omega))} \leq  r, 
$$
the system \eqref{sistema2_thermistor} has a unique strong solution. 
We will show that
$$
\|y(t)\|_{H^3(\Omega)} + \|p(t)\|_{H^3(\Omega)} \leq \tilde C e^{-\rho t}, \quad \text{for} \quad \rho < \gamma, \quad t \geq 0, 
$$
where $\gamma>0$ is the decay constant appearing in \eqref{Hyp3}  and $\tilde C=\tilde C(\kappa_0,\kappa_1,\sigma_0,\sigma_1, M, \Omega, N, \rho, r)$.

Indeed, from \eqref{sistema2_thermistor} we define
$$
h := -\Delta y = \frac{1}{\kappa(y)} \left( -y_t + \kappa'(y) |\nabla y|^2 + \sigma(y) |\nabla p|^2 + 2\sigma(y) \nabla p \cdot \nabla z^* + \sigma(y) |\nabla z^*|^2 \right)$$
and
$$
k := -\Delta p = \frac{1}{\sigma(p)} \left( -p_t + \sigma'(y) \nabla y \cdot \nabla p + \sigma'(y) \nabla y \cdot \nabla z^* + \sigma(y) \Delta z^* - (z^*)_t \right).
$$

For the function $h$ we have
\begin{align*}
    \nabla h & = -\frac{\kappa'(y) \nabla y}{\kappa(y)^2} \left( -y_t + \kappa'(y) |\nabla y|^2  + \sigma(y) |\nabla p|^2 + 2\sigma(y) \nabla p \cdot \nabla z^* + \sigma(y) |\nabla z^*|^2 \right) \\
    &+ \frac{1}{\kappa(y)} \left( -\nabla y_t + \kappa''(y) \nabla y | \nabla y |^2 + 2\kappa'(y) \Delta y \nabla y + \sigma'(y) \nabla y |\nabla p|^2 + 2\sigma(y) \Delta p \nabla p  \right.\\
    &\left. +2\sigma'(y)\nabla y (\nabla p \cdot \nabla z^*) + 2\sigma(y) \Delta p \nabla z^* + 2\sigma(y) \Delta z^* \nabla p + \sigma'(y) \nabla y |\nabla z^*|^2 + 2 \sigma(y) \Delta z^* \nabla z^* \right) \\
    &= U_1 + U_2. 
\end{align*}

On the other hand, for $k$ we get
\begin{align*}
    \nabla k & =  -\frac{\sigma'(y) \nabla y}{\sigma(y)^2} \left( -z_t + \sigma'(y) \nabla y \cdot \nabla p + \sigma'(y) \nabla y \cdot \nabla z^* + \sigma(y)\Delta z^* - (z^*)_t \right) \\
    &+ \frac{1}{\sigma(y)} \left( -\nabla p_t + \sigma''(y) \nabla y ( \nabla y \cdot \nabla p) + \sigma'(y) \Delta y \nabla p + \sigma'(y) \Delta p \nabla y  \right.\\
    &\left. + \sigma''(y) \nabla y (\nabla y \cdot \nabla z^*) + \sigma'(y)\Delta y \nabla z^* + 2 \sigma'(y) \Delta z^* \nabla y + \sigma(y) \nabla(\Delta z^*)  - (\nabla z^*)_t \right) \\
    &= V_1 + V_2. 
\end{align*}

Using Holder's inequality and by the immersion $H^1(\Omega) \hookrightarrow L^6(\Omega)$, which is valid for $N \leq 3$, we have
\begin{equation}\label{eq:U_1}
\begin{split}
    \int_\Omega |U_1|^2 dx & \leq C \int_\Omega \left( |\nabla y|^2 |y_t|^2 + |\nabla y|^6 + |\nabla y|^2 |\nabla p|^4 + |\nabla y|^2 |\nabla p|^2 |\nabla z^*|^2 + |\nabla y|^2 |\nabla z^*|^4 \right) dx \\
    & \leq C \left( \| \nabla y \|_{L^4(\Omega)}^2 \| y_t \|_{L^4(\Omega)}^2 + \| \nabla y \|_{L^6(\Omega)}^6 +  \| \nabla y \|_{L^6(\Omega)}^2 \| \nabla p \|_{L^6(\Omega)}^4 \right.\\
    &\left. + \| \nabla y \|_{L^6(\Omega)}^2 \| \nabla p \|_{L^6(\Omega)}^2 \| \nabla z^* \|_{L^6(\Omega)}^2 + \| \nabla y \|_{L^6(\Omega)}^2 \| \nabla z^* \|_{L^6(\Omega)}^4 \right) \\
    & \leq C \left( \| \Delta y \|^2 \| \nabla y_t \|^2 + \| \Delta y \|^6 + \| \Delta y \|^2 \| \Delta p \|^4  + \| \Delta y\|^2 \| \Delta p \|^2 \| \Delta z^* \|^2 \right. \\
    &\left. + \| \Delta y \|^2 \| \Delta z^* \|^4  \right),
\end{split}
\end{equation}
and
\begin{equation}\label{eq:U_2}
\begin{split}
    \int_\Omega |U_2|^2 dx &\leq C \int_\Omega \left( |\nabla y_t|^2 + |\nabla y|^6 + |\Delta y|^2 |\nabla y|^2 + |\nabla y|^2 |\nabla p|^4 + |\Delta p|^2 |\nabla p|^2 \right. \\
    &\left. + |\nabla y|^2 |\nabla p|^2 |\nabla z^*|^2 + |\Delta p|^2 |\nabla z^*|^2 + |\Delta z^*|^2 |\nabla p|^2 + |\nabla y|^2 |\nabla z^*|^4 + |\Delta z^*|^2 |\nabla z^*|^2 \right) \\
    &\leq C \left( \|\nabla y_t\|_{L^2(\Omega)}^2 + \|\nabla y\|_{L^6(\Omega)}^6 + U_{2,3} 
    + \|\nabla y\|_{L^6(\Omega)}^2\|\nabla p\|_{L^6(\Omega)}^4 + U_{2,5} \right. \\
    & \left. + \|\nabla y\|_{L^6(\Omega)}^2 \|\nabla p\|_{L^6(\Omega)}^2 \|\nabla z^* \|_{L^6(\Omega)}^2 + U_{2,7}
    + \|\Delta z^* \|_{L_4(\Omega)}^2 \| \nabla p \|_{L^4(\Omega)}^2 \right. \\
    & \left. + \|\nabla y\|_{L^6(\Omega)}^2 \|\nabla z^* \|_{L^6(\Omega)}^4 + \|\Delta z^*\|_{L^4(\Omega)}^2 \|\nabla z^*\|_{L^4(\Omega)}^2 \right) \\
    & \leq C \left(  \|\nabla y_t\|^2 + \|\Delta y\|^6 
    + U_{2,3} 
    + \|\Delta y\|^2 \|\Delta p\|^4 + U_{2,5} 
    + \|\Delta y\|^2 \|\Delta p\|^2 \|\Delta z^* \|^2 + U_{2,7} \right. \\
    & \left. + \|\nabla\Delta z^* \|^2 \| \Delta p \|^2 
      + \|\Delta y\|^2 \|\Delta z^* \|^4 + \|\nabla \Delta z^*\|^2 \|\Delta z^*\|^2 \right).
\end{split}
\end{equation}    

Since we do not have estimates in $H^3$ of $y$ and $p$, in order to estimate the terms $U_{2,k}$, $k=3,5,7$, 
following \cite{Thermistor_HNLP-23}, we consider the immersion $H^1(\Omega) \hookrightarrow L^6(\Omega) \hookrightarrow L^{\frac{2p}{2-p}}(\Omega)$, for $1 \leq p \leq \frac{3}{2} < 2$ we have $\frac{2p}{2-p} \leq 6$. Then, using Hölder inequality with dual exponents $q=\frac{2}{p}$ and $q'=\frac{2}{2-p}$, we have, for $U_{2,3}$,
\begin{align*}
    \| \Delta y \nabla y  \|_{L^p(\Omega)}^p & \leq \int_\Omega |\Delta y|^p |\nabla y|^p \leq C \left(\int_\Omega |\Delta y|^2 \right)^{p/2} \left(\int_\Omega |\nabla y|^{\frac{2p}{2-p}} \right)^{\frac{2-p}{2}}\\
    &\leq C \|\Delta y\|^p \|\nabla y \|_{L^6(\Omega)}^p \leq C \| \Delta y\|^{2p}.
\end{align*}

Taking $p=\frac{3}{2}$ we have 
$$
\| \Delta y \nabla y  \|_{L^{3/2}(\Omega)} \leq C \|\Delta y\|^2.
$$

Analogously, for $U_{2,5}$ and $U_{2,7}$, we get
$$
\| \Delta p \nabla p  \|_{L^{3/2}(\Omega)} \leq C \|\Delta p\|^2, \quad
\| \Delta p \nabla z^*  \|_{L^{3/2}(\Omega)} \leq C \|\Delta p\|\|\Delta z^* \|.
$$

Then, as $L^2(\Omega) \hookrightarrow L^{3/2}(\Omega)$, using the estimates for $U_{2,k}$, $k=3,5,7$,
\begin{align*}
    \| U_2 \|_{L^{3/2}}^2 \leq & C \left( \|\nabla y_t\|^2 + \|\Delta y\|^6 + \|\Delta y\|^4 + \|\Delta y\|^2 \|\Delta p\|^4 + \|\Delta p\|^4 \right. \\
    & \left. + \|\Delta y\|^2 \|\Delta p\|^2 \|\Delta z^* \|^2 + \|\Delta p \|^2 \|\Delta z^* \|^2 + \|\Delta z^* \|^2 \| \Delta p \|^2 \right. \\
    & \left. + \|\Delta y\|^2 \|\Delta z^* \|^4 + \|\nabla \Delta z^*\|^2 \|\Delta z^*\|^2 \right).
\end{align*}

Thus,
\begin{align*}
    \|\nabla h\|_{L^{3/2}}^2 \leq &C (\|U_1\|_{L^{3/2}}^2 + \|U_2\|_{L^{3/2}}^2) \\
    \leq & C \left( \| \Delta y \|^2 \| \nabla y_t \|^2 + \| \Delta y \|^6 + \| \Delta y \|^2 \| \Delta p \|^4  + \| \Delta y\|^2 \| \Delta p \|^2 \| \Delta z^* \|^2 + \| \Delta y \|^2 \| \Delta z^* \|^4  \right) \\
    & + C \left( \|\nabla y_t\|^2 + \|\Delta y\|^6 + \|\Delta y\|^4 + \|\Delta y\|^2 \|\Delta p\|^4 + \|\Delta p\|^4 \right. \\
    &\left. \qquad + \|\Delta y\|^2 \|\Delta p\|^2 \|\Delta z^* \|^2 + \|\Delta p \|^2 \|\Delta z^* \|^2 + \|\Delta z^* \|^2 \| \Delta p \|^2 \right. \\
    &\left. \qquad + \|\Delta y\|^2 \|\Delta z^* \|^4 + \|\nabla \Delta z^*\|^2 \|\Delta z^*\|^2
    \right).   
\end{align*}

Since
$$
\|h\|^2 \leq C \left(\|y_t\|^2 + \|\Delta y\|^4 + \|\Delta p\|^4 + \|\Delta p\|^2 \|\Delta z^*\|^2 + \|\Delta z^*\|^4 \right),
$$
we get
\begin{align*}\label{h_W32}
    \|h\|_{W^{1,3/2}}^2 \leq & C \left( \| \Delta y \|^2 \| \nabla y_t \|^2 + \| \Delta y \|^6 + \| \Delta y \|^2 \| \Delta p \|^4  + \| \Delta y\|^2 \| \Delta p \|^2 \| \Delta z^* \|^2 + \| \Delta y \|^2 \| \Delta z^* \|^4  \right) \\
    & + \left( \|\nabla y_t\|^2 + \|\Delta y\|^6 + \|\Delta y\|^4 + \|\Delta y\|^2 \|\Delta p\|^4 + \|\Delta p\|^4 \right. \\
    &\left. + \|\Delta y\|^2 \|\Delta p\|^2 \|\Delta z^* \|^2 + \|\Delta p \|^2 \|\Delta z^* \|^2 + \|\Delta z^* \|^2 \| \Delta p \|^2 \right. \\
    &\left. + \|\Delta y\|^2 \|\Delta z^* \|^4 +  \|\nabla \Delta z^*\|^2 \|\Delta z^*\|^2 + \|\Delta z^*\|^4 \right) \\
    &\leq C Q_{y,p}.
\end{align*}

Since $-\Delta y = h$ we have $y \in W^{3,3/2}(\Omega)$ and 
$$
    \|y\|_{W^{3,3/2}(\Omega)} \leq C \| h \|_{W^{1,3/2}(\Omega)}.
$$

Using the embedding $W^{1,3/2}(\Omega) \hookrightarrow L^3(\Omega)$ and $W^{2,3/2}(\Omega) \hookrightarrow L^q(\Omega)$ for $1 \leq q < \infty$, we have
\begin{equation}\label{eq:est_L2}
\begin{split}
    \| \Delta y \nabla y \|^2 &\leq C \| \Delta y \|_{L^3(\Omega)}^2 \| \nabla y\|_{L^6}^2 \\
    &\leq C \| \Delta y \|_{W^{1,3/2}(\Omega)}^2 \| \Delta y \|^2 \\
    &\leq C \|h\|_{W^{1,3/2}(\Omega)}^2 \|\Delta y\|^2 \\
    &\leq C Q_{y,p} \|\Delta y\|^2.
\end{split}
\end{equation}

Analogously, we get
$$
    \| \Delta y \nabla p \|^2 \leq C Q_{y,p} \|\Delta p\|^2, \qquad \| \Delta y \nabla z^* \|^2 \leq C Q_{y,p} \|\Delta z^*\|^2. 
$$

On the other hand, to estimate $\|\nabla k\|^2$, we compute, analogously,   
\begin{equation}\label{eq:V_1}
\begin{split}
    \int_\Omega |V_1|^2 dx & \leq C \int_\Omega \left( |\nabla y|^2 |z_t|^2 + |\nabla y|^4 |\nabla p|^2 + |\nabla y|^4 |\nabla z^*|^2 + |\nabla y|^2 |\Delta z^*|^2 + |\nabla y|^2 |(z^*)_t|^2 \right) dx \\
    & \leq C \left( \| \nabla y \|_{L^4(\Omega)}^2 \| z_t \|_{L^4(\Omega)}^2 + \| \nabla y \|_{L^6(\Omega)}^2 \| \nabla y \|_{L^6(\Omega)}^2 \| \nabla p \|_{L^6(\Omega)}^2 \right.\\
    &\left. + \| \nabla y \|_{L^6(\Omega)}^2 \| \nabla y \|_{L^6(\Omega)}^2 \| \nabla z^* \|_{L^6(\Omega)}^2 + \| \nabla y \|_{L^4(\Omega)}^2 \| \Delta z^* \|_{L^4(\Omega)}^2 + \| \nabla y \|_{L^4(\Omega)}^2 \| (z^*)_t \|_{L^4(\Omega)}^2 \right) \\
    & \leq C \left( \| \Delta y \|^2 \| \nabla z_t \|^2 + \| \Delta y \|^4 \| \Delta p \|^2 + \| \Delta y \|^4 \| \Delta z^* \|^2 + \| \Delta y\|^2 \|\nabla (\Delta z^*) \|^2 \right. \\
    & \left. + \| \Delta y \|^2 \| \nabla (z^*)_t \|^2  \right),
\end{split}
\end{equation}
and
\begin{equation}\label{eq:V_2}
\begin{split}
    \int_\Omega |V_2|^2 dx & \leq C \int_\Omega \left( |\nabla p_t|^2 + |\nabla y|^4 |\nabla p|^2 + |\Delta y|^2 |\nabla p|^2 + |\Delta p|^2 |\nabla y|^2 + |\nabla y|^4 |\nabla z^*|^2 \right. \\
    &\left. + |\Delta y|^2 |\nabla z^*|^2 + |\Delta z^*|^2 |\nabla y|^2 + |\nabla(\Delta z^*)|^2 + |\nabla (z^*)_t|^2 \right) dx \\
    & \leq C \left( \| \nabla p_t \|_{L^2(\Omega)}^2 + \| \nabla y \|_{L^6(\Omega)}^4 \| \nabla p \|_{L^6(\Omega)}^2 + V_{2,3} + V_{2,4} 
    + \| \nabla y \|_{L^6(\Omega)}^4 \| \nabla z^* \|_{L^6(\Omega)}^2 \right. \\
    &\left. + V_{2,6} + \| \Delta z^* \|_{L^4(\Omega)}^2 \| \nabla y \|_{L^4(\Omega)}^2 + \| \nabla (\Delta z^*) \|_{L^2(\Omega)}^2 + \| \nabla (z^*)_t \|_{L^2(\Omega)}^2 \right) \\
    & \leq C \left( \| \nabla p_t \|^2 + \| \Delta y \|^4 \| \Delta p \|^2 + V_{2,3} + V_{2,4}  + \| \Delta y \|^4 \| \Delta z^* \|^2 + V_{2,6}
    \right.\\
    &\left. 
    + \| \nabla(\Delta z^*) \|^2 \| \Delta y \|^2 + \| \nabla (\Delta z^*) \|^2 + \| \nabla (z^*)_t \|^2 \right),
\end{split}
\end{equation}
where $V_{2,3} = \|\Delta y \nabla p\|^2$, $V_{2,4}= \|\Delta p \nabla y\|^2$ and  $V_{2,6}=\|\Delta y \nabla z^*\|^2$. Then, proceeding as above for $U_{2,3}$ we get
$$
    \|\Delta y \nabla p\|_{L^{3/2}} \leq C \|\Delta y\| \|\Delta p\| , \quad \|\Delta p \nabla y\|_{L^{3/2}} \leq C \|\Delta p\| \|\Delta y\| , \qquad  \|\Delta y \nabla z^*\| _{L^{3/2}} \leq C \|\Delta y\| \|\Delta z^*\|.
$$

Thus,
\begin{equation*}
    \begin{split}
        \|\nabla k\|_{L^{3/2}(\Omega)}^2 & \leq C \left( \|V_1\|_{L^{3/2}(\Omega)}^2 + \|V_{2}\|_{L^{3/2}(\Omega)}^2 \right) \\
        & \leq C \left( \| \Delta y \|^2 \| \nabla z_t \|^2 + \| \Delta y \|^4 \| \Delta p \|^2 + \| \Delta y \|^4 \| \Delta z^* \|^2 + \| \Delta y\|^2 \|\nabla (\Delta z^*) \|^2 \right. \\
        & \left. \qquad + \| \Delta y \|^2 \| \nabla (z^*)_t \|^2 \right) \\
        & + C \left(
        \| \nabla p_t \|^2 + \| \Delta y \|^4 \| \Delta p \|^2 + \|\Delta y\|^2 \| \Delta p \|^2 + \| \Delta p \|^2 \|\Delta y\|^2  
        \right.\\
        &\left. \qquad + \| \Delta y \|^4 \| \Delta z^* \|^2 + \| \Delta y \|^2 \| \Delta z^* \|^2
        + \| \nabla(\Delta z^*) \|^2 \| \Delta y \|^2 + \| \nabla (\Delta z^*) \|^2 + \| \nabla (z^*)_t \|^2
        \right) 
    \end{split}
\end{equation*}
Moreover,
$$\|k\|^2 \leq C \left( \|p_t\|^2 + \|\Delta y \|^2 \|\Delta p \|^2 + \|\Delta y \|^2 \| \Delta z^* \|^2 + \| \Delta z^* \|^2 + \| z^*_t \|^2 \right).$$ 
Therefore, 
\begin{equation*}
    \begin{split}
        \|k\|_{W^{1,3/2}(\Omega)}^2 & \leq 
        C \left( \| \Delta y \|^2 \| \nabla z_t \|^2 + \| \Delta y \|^4 \| \Delta p \|^2 + \| \Delta y \|^4 \| \Delta z^* \|^2 + \| \Delta y\|^2 \|\nabla (\Delta z^*) \|^2 \right. \\
        & \left. \qquad + \| \Delta y \|^2 \| \nabla (z^*)_t \|^2 \right) \\
        & + C \left(
        \| \nabla p_t \|^2 + \| \Delta y \|^4 \| \Delta p \|^2 + \|\Delta y\|^2 \| \Delta p \|^2 + \| \Delta p \|^2 \|\Delta y\|^2    
        \right.\\
        &\left. \qquad + \| \Delta y \|^4 \| \Delta z^* \|^2 + \| \Delta y \|^2 \| \Delta z^* \|^2
        + \| \nabla(\Delta z^*) \|^2 \| \Delta y \|^2 + \| \nabla (\Delta z^*) \|^2 + \| \nabla (z^*)_t \|^2 \right) \\
        & \leq C \hat Q_{y,p}.
    \end{split}
\end{equation*}

Thus, proceeding as in \eqref{eq:est_L2}, we get
$$
V_{2,3} \leq C Q_{y,p} \|\Delta p\|^2, \qquad V_{2,4} \leq C \hat Q_{y,p} \|\Delta y\|^2, \qquad V_{2,6} \leq C Q_{y,p} \|\Delta z^*\|^2.
$$

Therefore, using \eqref{eq:U_1}, \eqref{eq:U_2}, \eqref{eq:V_1} and \eqref{eq:V_2}, we have that
$\|h\|^2 + \|\nabla h\|^2 \leq C Q_{y,p} + U_{2,3} + U_{2,5} + U_{2,7}$ and 
$\|k\|^2 + \|\nabla k\|^2 \leq C \hat Q_{y,p} + V_{2,3} + V_{2,4} + V_{2,6}$. Thus,
\begin{equation*}
    \begin{split}
        \|h\|^2 + \|\nabla h\|^2 + \|k\|^2 + \|\nabla k\|^2 \leq & C Q_{y,p} + C Q_{y,p} \|\Delta y\|^2 + C Q_{y,p} \|\Delta p\|^2 + C Q_{y,p} \|\Delta z^*\|^2 \\
        &+ C \hat Q_{y,p} + C Q_{y,p} \|\Delta p\|^2 + C \hat Q_{y,p}\|\Delta y \|^2 + C Q_{y,p} \|\Delta z^*\|^2.   
    \end{split}
\end{equation*}

In other words, we proved that
\begin{equation}\label{eq:est_z_H3}
\begin{split}
    \| y \|_{H^3(\Omega)}^2 + \| p \|_{H^3(\Omega)}^2 &\leq C \left( \|h\|^2 + \|\nabla h\|^2 + \|k\|^2 + \|\nabla k\|^2 \right) \\
    &\leq C \left( 1 + \|\Delta z\|^2 + \|\Delta p\|^2 + \|\Delta z^*\|^2
    \right) (Q_{y,p} +  \hat Q_{y,z}).
\end{split}
\end{equation}

As in \cite{Thermistor_HNLP-23}, we show that each term of the right-hand side of \eqref{eq:est_z_H3} has an exponential decay in $t$.
If we denote
$$S(t) = \|y(t) \|^2 + \|\nabla y(t)\|^2 + \|\Delta y(t)\|^2 + \|\nabla y_t(t)\|^2 + \|p(t) \|^2 + \|\nabla p(t)\|^2 + \|\Delta p(t)\|^2 + \|\nabla p_t(t)\|^2,$$
from the estimate \eqref{A-eq8} in Theorem \ref{Theo1} with $f=g=0$ we deduce that, for some $\rho>0$,
$$
\frac{d}{dt} S(t) + \rho S(t) \leq C \left( \|z^*(t)\|_{H^2(\Omega)}^2 + \|z^*_t(t)\|_{H^2(\Omega)}^2 + \| z^*(t) \|_{H^2(\Omega)}^4 + \|z^*_t(t)\|^4_{H^2(\Omega)} \right) \leq C e^{-\gamma t}.
$$

If $\rho < \gamma$ we get
$S(t) \leq C P_0 e^{-\rho t}$, 
where $P_0 = \frac{1}{\gamma-\rho} + r^2$.

Thus, we have shown that
\begin{equation}\label{eq:decay_z_H3}
    \begin{split}
        \| z(t) \|_{H^3(\Omega)}^2 + \| p(t) \|_{H^3(\Omega)}^2 & \leq C (S(t)+S^2(t)+S^3(t)+S^4(t)) \\
        &\leq  C\left( 1 + P_0 + P_0^2 + P_0^3 + P_0^4 \right) e^{-\rho t}\\
        &= \tilde{C} e^{-\rho t},
    \end{split}
\end{equation}
where $\tilde{C} = C\left( 1 + P_0 + P_0^2 + P_0^3 + P_0^4 \right)$ \ with $C= C(\kappa_0,\kappa_1,\sigma_0,\sigma_1, M, N, \Omega, \rho, r)$.
\end{proof}


\section{Controllability of the system \eqref{sistema2_thermistor}}\label{sec4:controllability}

To prove Theorem \ref{maintheorem}, we employ a technique called Liusternik's Inverse Mapping Theorem in Banach spaces (see \cite{Alekseev}). 
\begin{teo}[Liusternik]\label{Liusternik}
Let $ \mathcal{Y}$ and $ \mathcal{Z}$ be Banach spaces and let $\mathcal{A}:B_{r}(0)\subset  \mathcal{Y}\rightarrow  \mathcal{Z}$ be a $\mathcal{C}^{1}$ mapping. Let as assume that $\mathcal{A}^{\prime}(0)$ is onto and let us set $\mathcal{A}(0)=\zeta_{0}$. Then, there exist $\delta >0$, a mapping $W: B_{\delta}(\zeta_{0})\subset  \mathcal{Z}\rightarrow  \mathcal{Y}$ and a constant $K>0$ such that
\begin{equation*}
    W(z)\in B_{r}(0),\quad \mathcal{A}(W(z))=z\quad \text{and}\quad \Vert W(z)\Vert_{ \mathcal{Y}}\leq K\Vert z-\zeta_{0}\Vert_{ \mathcal{Z}}, \quad \forall\, z\in B_{\delta}(\zeta_{0}).
\end{equation*}
Here $B_{r}(0)$ and $B_{\delta}(\zeta_{0})$ are open balls of radius $r$ and $\delta$, respectively.\\
In particular, $W$ is a local inverse-to-the-right of $\mathcal{A}$.
\end{teo}
\begin{proof}
    See \cite{Alekseev}.
\end{proof}

The argument is inspired by \cite{Fursikov_Imanuvilov-96, Imanuvilov_Yamamoto-01}.
We prove Theorem \ref{maintheorem} for system \eqref{sistema2}, taking into account that an analogous proof works also for system \eqref{sistema2_control2}.

\subsection{The linearized system of (\ref{sistema2})}
\label{subsec:linearized_system}

The linearized system of \eqref{sistema2} about zero is given by
\begin{equation}\label{sistema2_linear}
    \left\{
    \begin{array}{ll}
        y_t - \nabla \cdot (\kappa(0) \nabla y) =  2\sigma(0) \nabla p \cdot \nabla z^* + \sigma'(0)  |\nabla z^*|^2 y  + f + v 1_\omega, \quad &\text{in}\quad Q_T\\
        p_t - \nabla \cdot (\sigma(0) \nabla p) - \nabla \cdot ((\sigma'(0) \nabla z^*)y) =  g  , \quad &\text{in}\quad Q_T, \\
        y=0, \quad p=0, \quad &\text{on}\quad \Sigma_T, \\
        y(x,0) = y_0(x), \ \ p(x,0)=p_0(x) \quad &\text{in}\quad \Omega.
    \end{array}
    \right.
\end{equation}
The adjoint system for \eqref{sistema2_linear} is
\begin{equation}\label{sistema_adj}
    \left\{
    \begin{array}{ll}
        -\phi_t - \nabla \cdot (\kappa(0) \nabla \phi) - \sigma'(0)|\nabla z^*|^2 \phi +  \sigma'(0) \nabla \psi \cdot \nabla z^* = F, \quad &\text{in}\quad Q_T\\
        -\psi_t - \nabla \cdot (\sigma(0) \nabla \psi) + 2 \nabla \cdot ((\sigma(0) \nabla z^*)\phi) =  G , \quad &\text{in}\quad Q_T, \\
        \phi=0, \quad \psi=0, \quad &\text{on}\quad \Sigma_T, \\
        \phi(x,T) = \phi_T(x), \ \  \psi(x,T) = \psi_T(x)  \quad &\text{in} \quad\Omega.
    \end{array}
    \right.
\end{equation}

Following well-known ideas, for suitable functions $f$ and $g$, the (global) null controllability of \eqref{sistema2_linear}
is obtained as a consequence of Carleman estimates for \eqref{sistema_adj}.

\subsection{Carleman estimates of system (\ref{sistema_adj})}\label{sec:carleman_estimates}

We will establish a suitable Carleman estimate for our adjoint system \eqref{sistema_adj}. We need the following result:
\begin{lem}[Fursikov's function]
Let $\omega \subset  \Omega$ be a non empty open set. Then there exists a function $\eta_0 \in C^2(\overline \Omega)$ such that $\eta_0 > 0$ in $\overline{\Omega}$ and $\vert \nabla \eta_0 \vert >0$ in $\Omega \setminus \omega$. 
\end{lem}
\begin{proof}
    Let us define $\hat \Omega \subset \mathbb{R}^N$ open set such that $\Omega \subset \hat \Omega$ and $\omega \subset \subset \hat \Omega$, then apply Lemma $1.1$ from \cite{Fursikov_Imanuvilov-96} to $\hat \Omega$ and $\omega$. Thus, there exists a function $\hat \eta_0 \in C^{2}(\overline {\hat \Omega})$ such that $\hat \eta_0 > 0$ in ${\hat \Omega}$, $\hat \eta_0 = 0$ in $\partial \hat \Omega$ and $\vert \nabla \hat \eta_0 \vert >0$ in $\hat\Omega \setminus \omega$. Taking $\eta_0 = \hat{ \eta}_0 \Big|_{\Omega}$ we conclude the proof.
\end{proof}

For some parameter $\lambda >0$ and for $K \geq \| \eta_0 \|_{\infty}$, we define the weight functions
$$ \alpha (x,t) = \frac{e^{2\lambda K} - e^{\lambda \eta_0(x)}}{t(T-t)}, \ \ \ \ \ \    \xi (x,t) = \frac{e^{\lambda \eta_0(x)}}{t(T-t)}.$$
We use the notation
\begin{align}\label{eq:formula_I}
I(\tau, \varphi) :=  \iint_{Q_T}  (s\xi)^{\tau-1} e^{-2s\alpha} \Big( \vert \varphi_t\vert^2 +\vert \Delta \varphi \vert^2  + (s \lambda \xi)^2 \vert \nabla \varphi \vert^2 + (s \lambda \xi)^4 \vert \varphi \vert^2 \Big) dx dt,  
\end{align} 
where $s, \lambda > 0$ and $\tau \in \mathbb{R}$.\\
We have the following Carleman estimate, which is straightforwardly adapted from \cite{GlobalCarleman}.
\begin{prop}\label{prop2.2}
Let $\omega \subset \subset \Omega$ be a non empty open set, there exist positive constants $C_0$, $s_0$ and $k_0$ such that, for $s>s_0$, $\lambda > \lambda_0$, $H \in [W^{1,\infty}(Q_T)]^{N^2}$, $A \in [L^{\infty}(Q_T)]^N$, $a \in L^{\infty}(Q_T)$, $F \in L^2(Q_T)$ and $\varphi_T\in L^2(\Omega)$, the solution of the equation 
\begin{equation} \label{sistema4}
    \left\{
    \begin{array}{ll}
        \displaystyle -\varphi_t - \nabla \cdot ( H \nabla \varphi ) + A \cdot \nabla \varphi + a \varphi= F, \ \ &\mbox{ in } \ Q_T, \\
        \varphi(x,t) =  0 ,\ \ &\mbox{ on } \ \ \Sigma_T,\\
        \displaystyle \varphi(x,T) = \varphi_T(x), &\mbox{ in } \ \Omega.
    \end{array}
    \right.
\end{equation}
satisfy
\begin{equation}
I(\tau, \varphi)\leq C_0 \left[ \iint_{Q_T} e^{-2s\alpha} (s \xi)^{\tau} \vert F \vert^2 dx dt +  \lambda^4 \iint_{\omega \times (0,T)} e^{-2s\alpha} (s \xi)^{\tau + 3} \vert \varphi \vert^2 dxdt \right].
\end{equation}

Moreover, $s_0$ has the form $s_0 = C_0(T+ T^2),$ where $C_0$ is a positive constant depending only on $\omega$ and $\Omega$.
\end{prop}
\begin{proof}
    See \cite{Imanuvilov_Yamamoto-01}.
\end{proof}
Let $\tau_1, \tau_2 \in \mathbb{R}$. A direct application of Proposition \ref{prop2.2} leads to
\begin{align*}
    &I(\tau_1, \phi) + I(\tau_2, \psi) \leq C \left(  \iint_{Q_T} e^{-2s\alpha} (s \xi)^{\tau_1} \vert F \vert^2 dx dt + \iint_{Q_T} e^{-2s\alpha} (s \xi)^{\tau_2} \vert G \vert^2 dx dt \right.\\ 
    &\phantom{AAAAAA}+ \left.  \lambda^4 \iint_{\omega \times (0,T)} e^{-2s\alpha} (s \xi)^{\tau_1 + 3} \vert \phi \vert^2 dxdt +  \lambda^4 \iint_{\omega \times (0,T)} e^{-2s\alpha} (s \xi)^{\tau_2 + 3} \vert \psi \vert^2 dxdt  \right).
\end{align*}

We also use the notations:
\begin{equation}
\label{eq:alpha_sigma_mais_menos}
\begin{split}
\alpha^+(t) = \max_{x\in \Omega}\alpha(x,t), \qquad \alpha^-(t) = \min_{x\in \Omega}\alpha(x,t),\\
\xi^+(t) = \max_{x\in \Omega}\xi(x,t),\ \ \qquad \xi^-(t) = \min_{x\in \Omega}\xi(x,t),\\
\hat{\alpha}(t) = 4 \alpha^{-}(t) - 3 \alpha^{+}(t),\  \ \  \  \ \  \ \  \  \  \  \  \  \  \  \ \    \\
\tau^* = \max \{ 2 \tau_2 - \tau_1 + 7, 1 - \tau_1, 4 \tau_2 - 3 \tau_1 + 15, \tau_2 + 7 \}.
\end{split}
\end{equation}

\begin{prop}\label{Prop_Benabdallah}
Assume the condition \eqref{Hyp8}. If $|\tau_1 - \tau_2| < 1$. Then, there exist three positive constants $\lambda^*, s^*, C^* = C^*(\Omega, \omega, t_1, t_2, T, \tau_1, \tau_2, z^*)$ and a constant $K \geq \| \eta_0\|_{\infty} + 2 \log(2) $, such that for every $(\phi_T, \psi_T) \in [L^2(\Omega)]^2$ the following Carleman estimate holds 
\begin{align*}
    I(\tau_1, \phi) + I(\tau_2, \psi) 
    &\leq C^* \left(  \iint_{Q_T} e^{-2s\alpha} (\lambda^4 (s \xi)^{3+\tau_2} \vert F \vert^2 + (s \xi)^{\tau_2} \vert G \vert^2 ) dx dt \right.\\ 
    &\phantom{AAAA}+ \left. \lambda^{16} \iint_{\omega \times (0,T)} e^{-2s\hat{\alpha}} (s \xi^{+})^{\tau^*} \vert \phi \vert^2 dxdt \right),
\end{align*}
for all $s \geq s^*$, $\lambda \geq \lambda^*$, for all $(\phi, \psi)$ solution of \eqref{sistema_adj}.
\end{prop}

\begin{proof}
    By hypothesis \eqref{Hyp8}, we have $\nabla z^*(x_0,t) \cdot \nu(x_0) \neq 0$ for all $t \in [t_1, t_2]$ and for some $x_0 \in \gamma \subset \partial \Omega \cap \partial \omega$. Therefore, we can apply Theorem 2.4 from \cite{Benab} for the functions $\phi \Big|_{\Omega \times [t_1, t_2]}$ and $\psi \Big|_{\Omega \times [t_1, t_2]}$. That is, we get
    \begin{align*}
    I(\tau_1, \phi, (t_1, t_2)) + I(\tau_2, \psi, (t_1, t_2)) 
    &\leq C \left(  \int_{t_1}^{t_2} \int_{\Omega} e^{-2s\alpha} (\lambda^4 (s \xi)^{3+\tau_2} \vert F \vert^2 + (s \xi)^{\tau_2} \vert G \vert^2 ) dx dt \right.\\ 
    &\phantom{AAAA}+ \left. \lambda^{16} \int_{\omega \times (t_1,t_2)} e^{-2s\hat{\alpha}} (s \xi^{+})^{\tau^*} \vert \phi \vert^2 dxdt \right),
\end{align*}
    where $C:=C(\Omega, \omega, T, \tau_1, \tau_2, z^*)>0$ and
    $$I(\tau, \varphi, (t_1,t_2)) :=  \int_{t_1}^{t_2} \int_{\Omega}  (s\xi)^{\tau-1} e^{-2s\alpha} \Big( \vert \varphi_t\vert^2 +\vert \Delta \varphi \vert^2  + (s \lambda \xi)^2 \vert \nabla \varphi \vert^2 + (s \lambda \xi)^4 \vert \varphi \vert^2 \Big) dx dt.$$

Our goal now is to show that
    $$I(\tau, \varphi) \leq \hat{C}(t_1, t_2, T) I(\tau, \varphi, (t_1, t_2)),$$
because then we can conclude that
\begin{align*}
    I(\tau_1, \phi) + I(\tau_2, \psi) 
    &\leq C^* \left(  \iint_{Q_T} e^{-2s\alpha} (\lambda^4 (s \xi)^{3+\tau_2} \vert F \vert^2 + (s \xi)^{\tau_2} \vert G \vert^2 ) dx dt \right.\\ 
    &\phantom{AAAA}+ \left. \lambda^{16} \iint_{\omega \times (0,T)} e^{-2s\hat{\alpha}} (s \xi^{+})^{\tau^*} \vert \phi \vert^2 dxdt \right),
\end{align*}
where $C^*:=C^*(\Omega, \omega, t_1, t_2, T, \tau_1, \tau_2, z^*)>0$.\\
Indeed, choosing $t_1$ and $t_2$ such that: for any $a \in \mathbb{R}$, $x \in \Omega$, $t \in (t_1, t_2)$ and $t' \in (0,T) \backslash [t_1, t_2]$ we have
    \begin{equation}\label{Eq-Carleman-1}
        (s \xi(x,t'))^{\tau + a}\ e^{-2s \alpha(x,t')} \leq (s \xi(x,t))^{\tau + a}\ e^{-2s \alpha(x,t)}.
    \end{equation}
This is possible due to the behavior of the function $(s\xi)^{\tau + a} e^{-2s \alpha}$ in $Q_T$. \\
We denote by $Q \varphi$ any operation applied to $\varphi$ that appears in $I(\tau, \varphi)$, that is, $Q \varphi$ can represent the following functions: $\varphi_t$, $\nabla \varphi$, $\Delta \varphi$ and $\varphi$.\\
    Multiplying $|Q \varphi(x,t')|^2$ to both sides of \eqref{Eq-Carleman-1} and integrating over $\Omega \times (0,T) \backslash [t_1, t_2]$ we have
    \begin{equation}\label{eq:para_Carl1}
    \begin{split}
        &\iint_{\Omega \times (0,T) \backslash [t_1, t_2]} (s \xi(x,t'))^{\tau + a}\ e^{-2s \alpha(x,t')} |Q \varphi(x,t')|^2 dxdt'\\ 
        &\phantom{AAAAAAA}\leq  \iint_{\Omega \times (0,T) \backslash [t_1, t_2]} (s \xi(x,t))^{\tau + a}\ e^{-2s \alpha(x,t)} |Q \varphi(x,t')|^2 dxdt'\\
        &\phantom{AAAAAAA}\leq  \sup_{\Omega \times [t_1, t_2]} \left( (s \xi(x,t))^{\tau + a}\ e^{-2s \alpha(x,t)} \right) \iint_{\Omega \times (0,T) \backslash [t_1, t_2]} |Q \varphi(x,t')|^2 dxdt'\\
        &\phantom{AAAAAAA}=  W_{\sup} \iint_{\Omega \times (0,T) \backslash [t_1, t_2]} |Q \varphi(x,t')|^2 dxdt'.
    \end{split}
    \end{equation}
    We define a bijection between $[0,t_1) \cup (t_2,T)$ and $(t_1, t_2)$ as
    $$t'(t) =\begin{cases}
        \dfrac{t_1 + t_2}{2} -\dfrac{t_2-t_1}{2t_1}\ t   &\text{if }\ \ t \in [0, t_1),\\
        t_2 - \dfrac{t_2-t_1}{2(T-t_2)}(t-t_2) &\text{if } \ \ t \in (t_2, T).
    \end{cases}$$
    Using the Theorem of Change of Variable, it is clear that
    \begin{equation}\label{eq:para_carl2}
        \iint_{\Omega \times (0,T) \backslash [t_1, t_2]} |Q \varphi(x,t')|^2 dxdt' \leq C(t_1, t_2, T) \iint_{\Omega \times (t_1, t_2)} |Q \varphi(x,t)|^2 dxdt.    
    \end{equation}
Combining \eqref{eq:para_Carl1} and \eqref{eq:para_carl2}, we obtain
\begin{align*}
    &\iint_{\Omega \times (0,T) \backslash [t_1, t_2]} (s \xi(x,t'))^{\tau + a}\ e^{-2s \alpha(x,t')} |Q \varphi(x,t')|^2 dxdt'\\ 
    &\phantom{AAAAAAA}\leq  W_{\sup}\ C(t_1, t_2, T) \iint_{\Omega \times (t_1, t_2)} |Q \varphi(x,t)|^2 dxdt\\
    &\phantom{AAAAAAA}\leq \dfrac{W_{\sup}}{W_{\inf}}\ C(t_1, t_2, T) \iint_{\Omega \times (t_1, t_2)} (s \xi(x,t))^{\tau + a}\ e^{-2s \alpha(x,t)} |Q \varphi(x,t)|^2 dxdt.
\end{align*}
Then,
\begin{align*}
    &\iint_{\Omega \times (0,T)} (s \xi(x,t))^{\tau + a}\ e^{-2s \alpha(x,t)} |Q \varphi(x,t)|^2 dxdt\\
    &\phantom{AAAAAA}=\iint_{\Omega \times (t_1,t_2)} (s \xi(x,t))^{\tau + a}\ e^{-2s \alpha(x,t)} |Q \varphi(x,t)|^2 dxdt\\
    &\phantom{AAAAAA}+\iint_{\Omega \times (0,T) \backslash [t_1, t_2]} (s \xi(x,t))^{\tau + a}\ e^{-2s \alpha(x,t)} |Q \varphi(x,t)|^2 dxdt\\
    &\phantom{AAAAAA}\leq \Big( 1 + \dfrac{W_{\sup}}{W_{\inf}}\ C(t_1, t_2, T) \Big) \iint_{\Omega \times (t_1, t_2)} (s \xi(x,t))^{\tau + a}\ e^{-2s \alpha(x,t)} |Q \varphi(x,t)|^2 dxdt.
\end{align*}
This concludes our desired result, 
$$I(\tau, \varphi) \leq \hat{C}(t_1, t_2, T) I(\tau, \varphi, (t_1, t_2)),$$
where $\hat{C}(t_1, t_2, T):=4\Big( 1 + \dfrac{W_{\sup}}{W_{\inf}}\ C(t_1, t_2, T) \Big)$.
\end{proof}

We will need some Carleman estimates for the solution to \eqref{sistema_adj} with weights that do not vanish at zero. So, let $m(\cdot)$ be a function satisfying
$$m \in C^{\infty}([0,T]),\ \ \ m(t)\geq \frac{T^2}{8}\ \ \text{in}\ \ \left[0,\frac{T}{2}\right],\ \ \ m(t)=t(T-t)\ \ \text{in}\ \ \left[\frac{T}{2},T\right],$$
and
$$\overline{\xi}(x,t):=\frac{e^{\lambda \eta_0(x)}}{m(t)},\ \ \ \overline{\alpha}(x,t):=\frac{e^{2\lambda K}-e^{\lambda \eta_0(x)}}{m(t)}.$$
If, we introduce the following notation
$$\overline{I}(\tau,\varphi):= \iint_{Q_T} (s \overline{\xi})^{\tau-1} e^{-2s \overline{\alpha}} \Big[ |\varphi_t|^2 + |\Delta \varphi|^2 + (s \lambda \overline{\xi})^2 |\nabla \varphi|^2 + (s \lambda\overline{\xi})^4 |\varphi|^2 \Big]\,dxdt,$$
and denote by
$$
\begin{array}{l}
\displaystyle (\overline{\xi})^{-}(t)=\min_{x\in \Omega} \overline{\xi}(x,t),\,\, \  \ \ \  (\overline{\xi})^{+}(t)=\max_{x\in \Omega} \overline{\xi}(x,t),\\
\displaystyle (\overline{\alpha})^{-}(t)=\min_{x\in \Omega} \overline{\alpha}(x,t),\ \ \ \,\,\, (\overline{\alpha})^{+}(t)=\max_{x\in \Omega}\overline{\alpha}(x,t),\\
\ \ \ \ \ \ \ \ \ \ \ \ \ \  (\overline{\hat{\alpha}})(t) = 4 (\overline{\alpha})^{-}(t) - 3 (\overline{\alpha})^{+}(t), 
\end{array}
$$
we have the following Carleman estimates:
\begin{prop}\label{Prop2}
 Under the same hypothesis of Proposition \ref{Prop_Benabdallah}, there exist positive constants $s_1$ and $C_1$ such that, for any $\lambda > > 1$, $s \geq s_1$ and $\varphi_T~\in~L^2(\Omega)$, the associated solution to \eqref{sistema_adj} satisfies
\begin{align*}
    \overline{I}(\tau_1, \phi) + \overline{I}(\tau_2, \psi) &\leq C_1 \left(  \iint_{Q_T} e^{-2s (\overline{\alpha})^{+}} \Big(\lambda^4( s (\overline{\xi})^{+})^{3 + \tau_2} \vert F \vert^2 + (s (\overline{\xi})^{+})^{\tau_2} \vert G \vert^2 \Big) dx dt \right.\\ 
    &\phantom{AAAA}+ \left. \lambda^{16} \iint_{\omega \times (0,T)} e^{-2s \overline{\hat{\alpha}}} (s (\overline{\xi})^{+})^{\tau^*} \vert \phi \vert^2 dxdt \right).
\end{align*}

Furthermore, $C_1$ and $s_1$ only depend on $\Omega$, $\omega$, $T$.
\end{prop}
\begin{proof}
    Similar to the proof of Proposition $2.2$ in \cite{Thermistor_HNLP-23}.
\end{proof}

\begin{rem}\label{remark-c}
If $\lambda > > 1$ and denoting $\displaystyle \beta(x,t)=\frac{1}{2}\overline{\alpha}(x,t)$, then
$$
\displaystyle \beta^+(t)=\max_{x\in \Omega}\beta(x,t) = \frac{(\overline{\alpha})^{+}(  t)}{2},\,\, \beta^{-}(t)=\min_{x\in \Omega}\beta(x,t) = \frac{(\overline{\alpha})^{-}(t)}{2},
$$
and
$$4 \beta^+(t) < 5 \beta^{-}(t) < \frac{75}{16} \beta^{+}(t), \ \forall t \in [0,T].$$
Finally, denoting
$$\overline{I}^{+}(\tau, \varphi):= \iint_{Q_T} (s (\overline{\xi})^{+})^{\tau -1} e^{-5s {\beta}^{+}} \Big[ |\varphi_t|^2 + |\Delta \varphi|^2 +(s  \lambda (\overline{\xi})^{+})^2 |\nabla \varphi|^2 + (s \lambda (\overline{\xi})^{+})^4 |\varphi|^2 \Big]dxdt,$$
then, due to Proposition \ref{Prop2} we have
\begin{align*}
    \overline{I}^+(\tau_1, \phi) + \overline{I}^+(\tau_2, \psi) &\leq C_1 \left( \iint_{Q_T} e^{-4s \beta^{+}} \Big( \lambda^4 (s (\overline{\xi})^{+})^{3 + \tau_2} \vert F \vert^2 + (s (\overline{\xi})^{+})^{\tau_2} \vert G \vert^2 \Big) dx dt \right.\\ 
    &\phantom{AAAA}+ \left. \lambda^{16} \iint_{\omega \times (0,T)} e^{-2s \overline{\hat{\alpha}}} (s (\overline{\xi})^{+})^{\tau^*} \vert \phi \vert^2 dxdt \right).
\end{align*}
\end{rem}

\begin{prop}\label{Observabillity}
    Under the assumptions of Proposition \ref{Prop2}, there exists $C>0$ such that 
    \begin{equation}\label{ineq_Observ}
        \int_\Omega (|\phi(0)|^2 + |\psi(0)|^2) dx \leq C \iint_{\omega \times (0,T)} |\phi|^2 dxdt.
    \end{equation}
\end{prop}

\begin{proof}
    Similar to the proof of Theorem $2.2$ of \cite{Benab}.
\end{proof}

\subsection{Null controllability for the linear system (\ref{sistema2_linear})}

In the sequel, we  fix  $\lambda = \lambda_1 > > 1$  and  $s = s_1$ and set
\begin{equation}\label{eq:weights1}
\rho(t) = e^{\frac{5s }{2}\beta^{-}(t)} ((\overline{\xi})^+(t))^{-(3-\tau^*)/2}, \,\, \rho_{0}(t) = e^{2s \beta^+(t)} ((\overline{\xi})^+(t))^{-(3 + \tau_2)/2},    
\end{equation}
\begin{equation}\label{eq:weights2}
\rho_{1}(t) = e^{s \overline{\hat{\alpha}}(t)} ((\overline{\xi})^{+}(t))^{-\tau^*/2},\,\, {\rho}_{k+2}(t) = e^{s \overline{\hat{\alpha}}(t)}((\overline{\xi})^{+}(t))^{-(\tau^* + 8+2k)/2}, \,\, 0 \leq k \leq 4.    
\end{equation}

It is clear that
\begin{equation}\label{eq:comp_weights}
\rho_{k+1} \leq C\rho_{k} \leq C \rho, \ 0 \leq k \leq 5  \ \ \text{ and } \  \ \rho_{k+2}\ \rho_{k+2,t} \leq C \rho_{k+1}^2, \ 0 \leq k \leq 4.    
\end{equation}

Thanks to Remark \ref{remark-c} and Proposition \ref{Observabillity}, we will be able to prove the null controllability of \eqref{sistema2_linear} for functions $f$ and $g$ that decay sufficiently  fast to zero as $t\to T^{-}$. Indeed, one has the following:

\begin{prop}\label{Prop3}
For any $T>0$, the linear system \eqref{sistema2_linear} is null controllable at time $T$. More precisely, for any $(y_0, p_0) \in [H_0^1(\Omega)]^2$ and if $(\rho f, \rho g) \in [L^2(Q_T)]^2$, then, there exist a control $v \in L^2(\omega \times (0,T))$ such that, if $(y,p)$ is the solution of \eqref{sistema2_linear}, one has
\begin{itemize}
	\item[a)] 
	\begin{equation}\label{eqa1-prop3}
	\iint_{Q_T} {\rho}_2^2 (|y|^2 + |p|^2)\ dxdt + \iint_{\omega \times (0,T)} \rho_0^2 |v|^2\, dxdt < +\infty.	
	\end{equation}
	 Furthermore, 
	 \begin{equation}\label{eqa2-prop3}
	 	\begin{array}{l}
	 	\displaystyle \iint_{\omega \times (0,T)} (|({\rho}_4 v)_t|^2 + |\Delta({\rho}_4 v)|^2)\,dxdt\\
		\noalign{\smallskip} 
	\phantom{DDDDD}
		\displaystyle \leq C \left( \iint_{Q_T} \rho^2 (|f|^2 + |g|^2)\ dxdt + \iint_{\omega \times (0,T)} \rho_0^2 |v|^2\ dxdt  \right).
	 	\end{array}
	 \end{equation}
	 
	\item[b)] 
	\begin{equation}\label{eqb-prop3}
\begin{array}{l}
\displaystyle \underset{t\in [0,T]}{\sup} \Big({\rho}_{3}^2(t) (||y(t)||^2 + ||p(t)||^2 ) \Big) + \iint_{Q_T} {\rho}^{2}_{3}(|\nabla y|^{2}+|\nabla p|^{2})\ dxdt\\
\noalign{\smallskip} 
	\phantom{DDDD}\displaystyle  \leq C \left(||y_0||^2 + ||p_0||^2 + \iint_{Q_T} {\rho}_2^2 (|y|^2+ |p|^2)\,dxdt\right. \\
\noalign{\smallskip} 
  \phantom{DDDDD} \displaystyle +\iint_{Q_T}\rho^{2}(|f|^{2}+ |g|^2)\,dx\,dt + \left. \iint_{\omega\times(0,T)} \rho_0^2 |v|^2 \,dxdt\right).
\end{array}
\end{equation}
	\item[c)]
	\begin{equation}\label{eqc-prop3}
\begin{array}{l}
\displaystyle \underset{t\in [0,T]}{\sup} \Big({\rho}_{4}^2(t) ( ||\nabla y(t) ||^2 + ||\nabla p(t)||^2 ) \Big) +\iint_{Q_T} {\rho}_{4}^2( |y_{t}|^{2} + | \Delta y|^2 + |p_t|^2 +|\Delta p|^{2})\,dxdt\\
\noalign{\smallskip} 
	\phantom{DDDDDD}
\displaystyle \leq C \left( ||y_0||_{H_0^1(\Omega)}^2 + ||p_0||_{H_0^1(\Omega)}^2 + \iint_{Q_T} {\rho}_2^2 (|y|^2+ |p|^2)\,dxdt \right. \\
\noalign{\smallskip} 
	\phantom{DDDDDDD}
\displaystyle
\left.  +\iint_{Q_T}\rho^{2} (|f|^{2}+ |g|^2)\,dxdt + \iint_{\omega\times(0,T)} \rho_0^2 |v|^2\,dxdt \right).
\end{array}
\end{equation}
\end{itemize} 
\end{prop}

\begin{proof}

\underline{Proof of a)}

Let us introduce the following constrained extremal problem:
\begin{equation}\label{p-extremal1}
\left\{\begin{array}{l}
\displaystyle \inf \left\{ \frac{1}{2}\left(\iint_{Q_T} {\rho}^{2}_{2}\left(|y|^{2}+|p|^{2}\right)\,dxdt+\iint_{\omega\times(0,T)}\rho^{2}_{0}|v|^{2}\,dxdt \right)\vspace{.3cm} \right\}\\
\mbox{subject to }\, v\in L^{2}(Q_T),\,\, \mathrm{supp}\, v\subset \omega\times(0,T)\,\, \mbox{and}\vspace{.3cm}\\

\left\{\begin{array}{ll}
    y_{t} -\kappa(0) \Delta y =2\sigma(0)\nabla z^*\cdot  \nabla p+\sigma'(0) |\nabla z^*|^{2}y+v 1_{\omega}+f, & \mbox{in}\,\,\, Q_T, \vspace{.1cm}\\

    \noalign{\smallskip} 
    p_t -\sigma(0)\Delta p-\nabla\cdot\left(\sigma'(0)\nabla z^* y\right)=g& \mbox{in}\,\,\,Q_T,\vspace{.1cm}\\
    \noalign{\smallskip} 
    y=0,\ \ \,p=0 &  \mbox{on}\,\,\,\Sigma_T,\\

    \noalign{\smallskip} 
    y(x,0)=y_{0}(x), \ \ p(x,0)=p_0(x) & \mbox{in}\,\,\,\Omega.
    \end{array}\right.
\end{array}\right.
\end{equation}
Let $\chi \in C^{\infty}_{0}(\omega),\,\, 0\leq \chi \leq 1,\,$ with $\chi|_{\omega_{0}}=1,$ we set
$$
\mathcal{P}_{0}=\Big\{(y,p)\in C^{\infty}_{0}(\overline{Q}_T)^{2}\ : \ y=0,\, p=0\, \  \mbox{on}\  \Sigma_T\Big\}
$$
and 
$$
\begin{array}{l}
    \displaystyle A\left((y,p);(\tilde{y},\tilde{p})\right) = \iint_{Q_T} {\rho}^{-2}_{2} \Big(L_1^{\ast}y + \sigma'(0)\nabla z^*\cdot \nabla p \Big) \Big(L_1^{\ast}\tilde{y} + \sigma'(0)\nabla z^*\cdot \nabla \tilde{p}  \Big)\, dxdt \vspace{.1cm}\\
				\noalign{\smallskip}\phantom{a\left((z,y);(\tilde{y},\tilde{z})\right)}
    \displaystyle +\iint_{Q_T} {\rho}^{-2}_{2} \Big(L_2^{\ast}p + 2\nabla\cdot(\sigma(0)\nabla z^* y)\Big) \Big(L_2^{\ast}p +2\nabla\cdot(\sigma(0)\nabla z^* \tilde{y})\Big)\, dxdt \vspace{.1cm}\\
    \noalign{\smallskip}\phantom{a\left((p,y);(\tilde{y},\tilde{p})\right)}
    \displaystyle +\iint_{\omega\times(0,T)}\chi \rho^{-2}_{0}y\tilde{y}\, dxdt, \ \  \,\, \forall\, (y,p),\,(\tilde{y},\tilde{p})\in \mathcal{P}_{0},
\end{array}
$$
where $L_1^{\ast}y=-y_{t}-\kappa(0)\Delta y-\sigma'(0)|\nabla z^*|^{2}y$\ \ and \ \  $L_2^{\ast}p = -p_t - \sigma(0) \Delta p$. 

Following the same  ideas as in \cite{Marin, Fursikov_Imanuvilov-96} and due to Proposition \ref{Prop2} we have that  $A(\cdot ; \cdot)$ defines an inner product.

Now, let us define $\mathcal{P}$ the completion of $\mathcal{P}_{0}$ with the inner product $A(\cdot;\cdot)$, that is,
$$|| (y, p) ||_{\mathcal{P}}^2 = A((y,p);(y,p)).$$
Let us define the linear functional $G:\mathcal{P} \to \mathbb{R}$ as
\begin{equation}\label{G-1}
G(y,p)=(y_{0},y(0))+(p_0,p(0))+\iint_{Q_T}(fy+gp)\, dxdt.
\end{equation}

By means of a simple calculation, we get
$$
\begin{array}{l}
G(y,p)\leq ||y_{0}|| ||y(0)|| + ||p_{0}|| ||p(0)||+||\rho f||_{L^{2}(Q_T)} ||\rho^{-1}y||_{L^{2}(Q_T)}+||\rho g||_{L^{2}(Q_T)}   ||\rho^{-1}p||_{L^{2}(Q_T)},
\end{array}
$$
and using the Carleman's inequality (Proposition \ref{Prop2} and Remark \ref{remark-c}) and  the Observability inequality (Proposition \ref{Observabillity}), we have
$$
G(y,p)\leq   C \Big(||y_{0}|| + ||p_{0}|| +||\rho f||_{L^{2}(Q_T)}+||\rho g||_{L^{2}(Q_T)} \Big)||(y,p)||_{\mathcal{P}},\,\, \forall\, (y,p)\in \mathcal{P}.
$$

Consequently $G$ is a bounded linear operator in $\mathcal{P}$. Then, in view of Lax-Milgram's Lemma, there exists one and only one $(\tilde{y},\tilde{p})$ satisfying
\begin{equation}\label{LML}
\left\{\begin{array}{l}
A((\tilde{y},\tilde{p});(y,p))=G(y,p),\,\,\forall (y,p)\in \mathcal{P}, \vspace{.1cm}\\
				\noalign{\smallskip}
(\tilde{y},\tilde{p})\in \mathcal{P}.
\end{array}\right.
\end{equation}

We finally get the existence of $(\hat{y},\hat{p})$, just setting
\begin{equation}\label{Lag0}
    \left\{
    \begin{array}{ll}
        \displaystyle \hat{y} =  {\rho}^{-2}_{2} \Big(L_1^{\ast}\tilde{y}-\sigma'(0)\nabla z^*\cdot \nabla\tilde{p}\Big),\,\, &\mbox{in}\,\, Q_T, \vspace{.1cm}\\
        \noalign{\smallskip} 
        \displaystyle \hat{p}=  {\rho}^{-2}_{2} \Big(L_2^{\ast} \tilde{p} + 2\nabla\cdot\left(\sigma(0)\nabla z^*\ \tilde{y}\right) \Big),\,\, &\mbox{in}\,\, Q_T, \vspace{.1cm}\\
        \noalign{\smallskip} 
        \hat{v}=-\rho^{-2}_{0}\tilde{y}\ 1_\omega,\,\, &\mbox{in}\,\, Q_T, \vspace{.1cm}\\
        \noalign{\smallskip} 
        \hat{y}=\hat{p}=0\,\, &\mbox{on}\,\,\Sigma_T.
    \end{array}
    \right.
\end{equation}

We see that $(\hat{y},\hat{p},\hat{v})$ solves \eqref{sistema2_linear} and since $(\hat{y},\hat{p})\in \mathcal{P}$ we have
$$
    \iint_{Q_T} {\rho}^{2}_{2} (|\hat{y}|^{2} + |\hat{p}|^{2})\,dxdt+\iint_{\omega\times(0,T)}\rho^{2}_{0}|\hat{v}|^{2}\,dxdt<+\infty.
$$
Furthermore, by Carleman's inequality (Remark \ref{remark-c}) we have
\begin{align}\label{Carl_y_p}
    \overline{I}^+(\tau_1, \tilde{y}) + \overline{I}^+(\tau_2, \tilde{p}) &\leq C \left( \iint_{Q_T} \rho_0^{-2} \vert L_1^{\ast}\tilde{y}-\sigma'(0)\nabla z^*\cdot \nabla\tilde{p} \vert^2 dxdt   \right. \nonumber\\ 
        &\phantom{A}+ \iint_{Q_T} \rho_0^{-2} \vert L_2^{\ast} \tilde{p} + 2\nabla \cdot (\sigma(0)\nabla z^* \tilde{y}) \vert^2  dx dt \\ 
        &\phantom{A}+ \left.  \iint_{\omega \times (0,T)} \rho_1^{-2}\vert \tilde{y} \vert^2 dxdt \right) \nonumber\\
        &\leq C \left( \iint_{Q_T} {\rho}_2^2 (|\hat{y}|^2 + |\hat{p}|^2) dxdt + \iint_{\omega \times (0,T)} {\rho}_0^{2}\vert \hat{v} \vert^2 dxdt \right). \nonumber
\end{align}    
For the functions $y^{\ast}= {\rho}_4 (\rho_0^{-2} \tilde{y})$ and $p^{\ast} = {\rho}_4 (\rho_0^{-2} \tilde{p})$, let us compute 
\begin{align*}
    L_1^{\ast}y^* &= -y^*_t - \kappa(0) \Delta y^* - \sigma'(0) |\nabla z^*|^2 y^*\\
    &= -({\rho}_4 \rho^{-2}_{0}\tilde{y})_t - \kappa(0) \Delta ({\rho}_4 \rho^{-2}_{0}\tilde{y}) - \sigma'(0) |\nabla z^*|^2 ({\rho}_4 \rho^{-2}_{0}\tilde{y})\\
    &= -({\rho}_4 \rho^{-2}_{0})_t\ \tilde{y} - {\rho}_4 \rho^{-2}_{0} \tilde{y}_t - \kappa(0) {\rho}_4 \rho^{-2}_{0} \Delta \tilde{y} - \sigma'(0) |\nabla z^*|^2 ({\rho}_4 \rho^{-2}_{0}\tilde{y})\\
    &= -({\rho}_4 \rho^{-2}_{0})_t\ \tilde{y} + {\rho}_4 \rho^{-2}_{0} L_1^* \tilde{y} 
\end{align*}
\begin{align*}
    \sigma'(0) \nabla z^* \cdot \nabla p^* = \sigma'(0) \nabla z^* \cdot \nabla  ({\rho}_4 (\rho_0^{-2} \tilde{p})) = {\rho}_4 \rho_0^{-2} (\sigma'(0) \nabla z^* \cdot \nabla   \tilde{p}).
\end{align*}
Then,
\begin{align*}
    f^* &:=L_1^{\ast}(y^*) - \sigma'(0) \nabla z^* \cdot \nabla p^*\\ 
        &= -({\rho}_4 \rho^{-2}_{0})_t\ \tilde{y} + {\rho}_4 \rho^{-2}_{0} L_1^* \tilde{y} - {\rho}_4 \rho_0^{-2} (\sigma'(0) \nabla z^* \cdot \nabla   \tilde{p})\\
        &= -({\rho}_4 \rho^{-2}_{0})_t\ \tilde{y} + {\rho}_4 \rho^{-2}_{0} \Big(L_1^* \tilde{y} - \sigma'(0) \nabla z^* \cdot \nabla   \tilde{p}\Big)\\
        &= -({\rho}_4 \rho^{-2}_{0})_t\ \tilde{y} + {\rho}_4 \rho^{-2}_{0} ({\rho}_2^2 \hat{y}).
\end{align*}
Analogously,
\begin{align*}
    L_2^*p^* = -p^*_t - \sigma(0) \Delta p^* 
    = -({\rho}_4 \rho_0^{-2})_t\ \tilde{p} + {\rho}_4 \rho_0^{-2} L_2^* \tilde{p},
\end{align*}
\begin{align*}
    \nabla \cdot (\sigma(0)\nabla z^* y^*) = \nabla \cdot (\sigma(0)\nabla z^* ({\rho}_4 \rho_0^{-2} \tilde{y})) = {\rho}_4 \rho_0^{-2}\ \nabla \cdot( \sigma(0) \nabla z^* \tilde{y}).
\end{align*}
Then
\begin{align*}
    g^* := L_2^*p^* + 2 \nabla \cdot (\sigma(0)\nabla z^* y^*)  
    = -({\rho}_4 \rho_0^{-2})_t\ \tilde{p} + {\rho}_4 \rho_0^{-2}({\rho}_2^2 \hat{p}).
\end{align*}
We have that $(y^{\ast},p^{\ast})$ solves the following system 
$$\left\{
    \begin{array}{lll}
        L_1^{\ast} y^{\ast} - \sigma'(0)\nabla z^* \cdot\nabla p^{\ast} = f^{\ast} &\mbox{in} &Q_T, \vspace{.1cm}\\
        L_2^*p^* + 2 \nabla \cdot (\sigma(0)\nabla z^* y^*) = g^{\ast} &\mbox{in} &Q_T, \vspace{.1cm}\\
        \noalign{\smallskip} 
        y^{\ast}=0, \  p^{\ast}=0 &\mbox{on} &\Sigma_T, \vspace{.1cm}\\
        \noalign{\smallskip} 
        y^{\ast}(x,T)=0, \  p^{\ast}(x,T)=0 & \mbox{in} &\Omega.
    \end{array}
\right.
$$
By a simple computation
\begin{align*}    
    {\rho}_4 \rho_0^{-2} \rho_2^2 
    &\leq C \rho_2 ({\rho_2}^{-2}) \rho_2^2 = C \rho_2.
\end{align*}
Then 
$$||{\rho}_4 \rho_0^{-2} \rho_2^2 \hat{y}||_{L^2(Q_T)} \leq C || \rho_2 \hat{y}||_{L^2(Q_T)} < +\infty,$$
$$||{\rho}_4 \rho_0^{-2} \rho_2^2 \hat{p}||_{L^2(Q_T)} \leq C || \rho_2 \hat{p}||_{L^2(Q_T)} < +\infty.$$
Furthermore,
\begin{align*}    
    |({\rho}_4 \rho_0^{-2} )_t| &= \left|\Big(e^{s \overline{\hat{\alpha}}} ((\overline{\xi})^+)^{-\frac{\tau^*}{2}-6} \ e^{-4s \beta^+} ((\overline{\xi})^+)^{\tau^*+3} \Big)_t \right|\\
                &= \left| \left( e^{-s(4\beta^+ - \overline{\hat{\alpha}})} ((\overline{\xi})^+)^{\frac{\tau^*}{2}-3} \right)_t \right| < C e^{-\frac{5}{2}s \beta^+} ((\overline{\xi})^+)^{\frac{-\tau^*+3}{2}}.
\end{align*}

Then, by \eqref{Carl_y_p} 
$$
    ||({\rho}_4 \rho_0^{-2} )_t \ \tilde{y}||_{L^2(Q_T)} \leq C \left( \iint_{Q_T} {\rho}_2^2 (|\hat{y}|^2 + |\hat{p}|^2) dxdt + \iint_{\omega \times (0,T)} {\rho}_0^{2}\vert \hat{v} \vert^2 dxdt \right),
$$
$$
    ||({\rho}_4 \rho_0^{-2} )_t \ \tilde{p}||_{L^2(Q_T)} \leq C \left( \iint_{Q_T} {\rho}_2^2 (|\hat{y}|^2 + |\hat{p}|^2) dxdt + \iint_{\omega \times (0,T)} {\rho}_0^{2}\vert \hat{v} \vert^2 dxdt \right).
$$
So, $f^{\ast}, g^{\ast} \in L^{2}(Q_T)$. Then, by regularity results in parabolic systems, we have
$$ y^{\ast}, p^{\ast}\in L^{2}(0,T;H^{2}(\Omega)),\,\, y^{\ast}_t, p^{\ast}_{t}\in L^{2}(0,T;L^{2}(\Omega)).$$
As $y^* 1_\omega =  -\rho_4 \hat{v} 1_{\omega}$, in particular
$${\rho}_{4}\hat{v} 1_{\omega}\in L^{2}(0,T;H^{2}(\Omega)),\,\, ({\rho}_{4}\hat{v})_{t} 1_{\omega} \in L^{2}(0,T;L^{2}(\Omega)),$$
with
$$
	\displaystyle \iint_{\omega \times (0,T)} (|({\rho}_4 v)_t|^2 + |\Delta({\rho}_4 v)|^2)\,dxdt \leq C  \left( \iint_{Q_T} {\rho}_2^2 (|\hat{y}|^2 + |\hat{p}|^2) dxdt + \iint_{\omega \times (0,T)} {\rho}_0^{2}\vert \hat{v} \vert^2 dxdt \right).
$$

\underline{Proof of b)}

Multiplying $\rho_3^2 y$ and $\rho_3^2 p$ the first and second equations of the linear system \eqref{sistema2_linear} and integrating in $Q_T$. Here we use \eqref{eq:comp_weights}, specifically:   $\rho_3 \rho_{3,t} \leq C \rho_2^2$. \\

\underline{Proof of c)}

Multiplying $-\rho_4^2 \Delta y$ and $-\rho_4^2 \Delta p$ the first and second equations of the linear system \eqref{sistema2_linear} and integrating in $Q_T$. Then, multiplying $\rho_4^2 y_t$ and $\rho_4^2 p_t$ at first and second equations of the linear system \eqref{sistema2_linear} and integrating in $Q_T$. Here we use \eqref{eq:comp_weights}, specifically: $\rho_4 \rho_{4,t} \leq C \rho_3^2$.

\end{proof}

\begin{prop}\label{Prop4}
Assuming the hypothesis in Proposition \ref{Prop3} and additionally assuming that $(y_0, p_0) \in [H^3(\Omega)]^2$, $({\rho}_5 f_t, \rho_5 g_t) \in~[L^2(Q_T)]^2$ and $(f(0), g(0)) \in [H_0^1(\Omega)]^2$, then the associated state $(y,p)$ satisfy
\begin{equation}
\label{eq1-Prop4}
\begin{array}{l}
 \displaystyle \iint_{Q_T} {\rho}^{2}_{6} (|y_{tt}|^{2}+|\Delta y_{t}|^{2}+|\nabla y_{t}|^{2}) \,dx\,dt +\underset{t\in [0,T]}{\sup} \Big( {\rho}^{2}_{5}(t) ||\Delta y(t)||^2 + {\rho}^{2}_{6}(t) ||\nabla y_{t}(t)||^2 \Big)  \\
\noalign{\smallskip} 
\displaystyle + \iint_{Q_T} {\rho}^{2}_{6} (|p_{tt}|^{2}+|\Delta p_{t}|^{2}+|\nabla p_{t}|^{2}) \,dx\,dt +\underset{t\in [0,T]}{\sup} \Big( {\rho}^{2}_{5}(t) ||\Delta p(t)||^2 + {\rho}^{2}_{6}(t) ||\nabla p_{t}(t)||^2 \Big)  \\ 
\noalign{\smallskip} 
 \phantom{DDDDSSSSS}
 \displaystyle \leq C \left( ||y_0||_{H^3(\Omega)}^2 + ||p_0||_{H^3(\Omega)}^2 + \iint_{Q_T} \rho_2^2 (|y|^2+|p|^{2}) \,dx\,dt \right.\\
\noalign{\smallskip} 
 \displaystyle \phantom{DDDDSSSSSSS} +\iint_{Q_T} \Big(\rho^2 (|f|^2 + |g|^2) +  {\rho}^{2}_{5} (|f_{t}|^{2} + |g_t|^2) \Big)\,dx\,dt \\
 \noalign{\smallskip} 
 \displaystyle \phantom{DDDDSSSSSSS} \left. +||f(0)||_{H^1_0(\Omega)}^2 + ||g(0)||_{H^1_0(\Omega)}^2 +\iint_{\omega\times(0,T)}  {\rho}^{2}_{0}|v|^{2}  \,dx\,dt \right). \\
\end{array}
\end{equation}
\end{prop}

\begin{proof}
In order to prove the estimate \eqref{eq1-Prop4}, let us derivative the first and second equation in \eqref{sistema2_linear} with respect to time variable $t$ in this way we have
\begin{equation}\label{est-ad1}
\begin{array}{l}
\displaystyle 
y_{tt}-\kappa(0)\Delta y_{t}=2\sigma(0)\nabla z_{0}\cdot\nabla p_{t}+2\left(\sigma(0)\nabla z_{0}\right)_{t}\cdot \nabla p+\left(\sigma(0)|\nabla z_{0}|^{2}y\right)_{t}+v_{t}\chi_{\omega}+f_{t}\\
\end{array}
\end{equation}
\begin{equation}\label{est-ad2}
\begin{array}{l}
p_{tt} - \sigma(0)\Delta p_{t}-\nabla\cdot\left( (\sigma'(0)\nabla z_{0})_{t}y\right)-\nabla\cdot\left( (\sigma'(0)\nabla z_{0})y_{t}\right)=g_{t}.
\end{array}
\end{equation}
Multiplying in  \eqref{est-ad1} by $\rho_5^2 y_t$ and in \eqref{est-ad2} by ${\rho}^{2}_{5} p_t$ and integrating in $Q_T$ and from Proposition \ref{Prop3}, we have
\begin{equation}\label{est-ad3}
    \begin{array}{lll}
        &\displaystyle \sup_{t\in[0,T]} \Big( {\rho}^{2}_{5}(t) (|| y_t(t)||^{2} + || p_t(t)||^{2}  ) \Big) + \iint_{Q_T} \rho_5^2 (|\nabla y_t|^2 + |\nabla p_t|^2) dxdt\\
        &\phantom{AAA}\displaystyle\leq C \left( || y_0 ||_{H^2(\Omega)}^2 + || p_0 ||_{H^2(\Omega)}^2 + \iint_{Q_T}  {\rho}^{2}_{5} (|f_{t}|^{2} + |g_t|^2) \,dx\,dt + (RS) \right),
    \end{array}
\end{equation}
here we use \eqref{eq:comp_weights}, specifically: $\rho_5 \rho_{5,t} \leq C \rho_4^2$ and $(RS)$ is the right side in \eqref{eqc-prop3}.

Multiplying in  \eqref{est-ad1} by $-\rho_6^2 \Delta y_t$ and in \eqref{est-ad2} by $-{\rho}^{2}_{5} \Delta p_t$ and integrating in $Q_T$ and from Proposition \ref{Prop3}, we have
\begin{equation}\label{est-ad4}
    \begin{array}{lll}
        &\displaystyle \sup_{t\in[0,T]} \Big( {\rho}^{2}_{6}(t) (|| \nabla y_t(t)||^{2} + || \nabla p_t(t)||^{2}  ) \Big) + \iint_{Q_T} \rho_6^2 (|\Delta y_t|^2 + |\Delta p_t|^2) dxdt\\
        &\phantom{AAAAAAAAAAAAAAAAAA}\leq C \Big( || \nabla y_t(0) ||^2 + || \nabla p_t(0) ||^2 +  (RS) \Big),
    \end{array}
\end{equation}
here we use \eqref{eq:comp_weights}, specifically: $\rho_6 \rho_{6,t} \leq C \rho_5^2$ and $(RS)$ is the right side in \eqref{est-ad3}.

Multiplying in  \eqref{est-ad1} by $\rho_6^2 y_{tt}$ and in \eqref{est-ad2} by ${\rho}^{2}_{6} p_{tt}$ and integrating in $Q_T$ and from Proposition \ref{Prop3}, we have
\begin{equation}\label{est-ad5}
    \begin{array}{lll}
        &\displaystyle \sup_{t\in[0,T]} \Big( {\rho}^{2}_{6}(t) (|| \nabla y_t(t)||^{2} + || \nabla p_t(t)||^{2}  ) \Big) + \iint_{Q_T} \rho_6^2 (|y_{tt}|^2 + |p_{tt}|^2) dxdt\\
        &\phantom{AAAAAAAAAAAAA}\leq C \Big( || \nabla y_t(0) ||^2 + || \nabla p_t(0) ||^2 +  (RS) \Big),
    \end{array}
\end{equation}
where we used \eqref{eq:comp_weights}, specifically: $\rho_6 \rho_{6,t} \leq C \rho_5^2$ and $(RS)$ is the right side in \eqref{est-ad3}.

For initial data estimates in \eqref{est-ad4} and \eqref{est-ad5}:
$$|| \nabla y_t(0) || + || \nabla p_t(0) || \leq C \Big( || y_0 ||_{H^3(\Omega)} + || p_0 ||_{H^3(\Omega)} + || \nabla v(0) 1_\omega || + ||f(0) ||_{H_0^1(\Omega)} + || g(0) ||_{H_0^1(\Omega)}\Big).$$
By Proposition \ref{Prop3}, we have $v 1_\omega \in \Big\{ u \in L^2(0,T/2;H^2(\Omega)) \ : \ u_t \in L^2(0,T/2;L^2(\Omega)) \Big\}$. So $v 1_\omega \in C([0,T/2];H^1(\Omega))$ and
$$|| \nabla v(0) 1_\omega || \leq C \left( \iint_{Q_T} \rho^2 (|f|^2 + |g|^2)\ dxdt + \iint_{\omega \times (0,T)} \rho_0^2 |v|^2\ dxdt  \right).$$

Finally, multiplying in \eqref{sistema2_linear} by $-{\rho}_{5}^2 \Delta y_t$ and $-{\rho}_{5}^2 \Delta p_t$ and integrating in $Q_T$, we have
\begin{equation}
    \begin{array}{lll}
    \displaystyle  \iint_{Q_T} {\rho}^{2}_{5 }\left(|\nabla y_{t}|^{2}+|\nabla p_{t}|^{2}\right)\,dxdt &+  \displaystyle\sup_{t\in[0,T]}\Big( {\rho}^{2}_{5}(t) (||\Delta y(t)||^{2} + ||\Delta p(t)||^{2}) \Big)\\
        &\leq C \Big( || y_0 ||_{H^2(\Omega)}^2 + || p_0 ||_{H^2(\Omega)}^2 +  (RS) \Big),
    \end{array}
\end{equation}
here we use \eqref{eq:comp_weights}, specifically: $\rho_5 \rho_{5,t} \leq C \rho_4^2$ and $(RS)$ is the right side in \eqref{eqc-prop3}.
\end{proof}

\subsection{Definition of map \texorpdfstring{$\mathcal{A}$}{A} and the Banach spaces \texorpdfstring{$\mathcal{Y}$}{Y} and \texorpdfstring{$\mathcal{Z}$}{Z}}

We are interested in using Theorem \ref{Liusternik}  (Liusternik's Theorem) to obtain our result. 

In order to do so, we define a map $\mathcal{A} : \mathcal{Y} \to \mathcal{Z}$ between suitable Banach spaces $\mathcal{Y}$ and $\mathcal{Z}$ whose definition and properties came from the controllability result of the linearized system and the additional estimates shown in Proposition \ref{Prop3}.

Denoting
\begin{align*}
    \mathcal{L}_1(y,p) & = y_t - \nabla \cdot (\kappa(0) \nabla y) -  2\sigma(0) \nabla p \cdot \nabla z^* - \sigma'(0) y |\nabla z^*|^2 - v 1_\omega,\\
    \mathcal{L}_2(y,p) & = p_t - \nabla \cdot (\sigma(0) \nabla p) - \nabla \cdot ((\sigma'(0) \nabla z^*)y),  
\end{align*}
let us define the space
\begin{equation}
\label{eq:espaceY}
\begin{array}{lll}
\mathcal{Y} &= \Big\{  (y,p,v)\in [L^{2}(Q_T)]^{2}\times L^{2}( \omega \times (0,T)) \ : 
\ y=p=0 \text{ on } \Sigma_T,  \\
&\phantom{AAA}\rho_{0} v\in L^{2}( \omega \times (0,T)), \ \rho_{2}y, \ \rho_{2}p \in L^2(Q_T),  \\
&\phantom{AAA}\rho \mathcal{L}_1(y,p), \ \rho \mathcal{L}_2(y,p) \in L^2(Q_T), \
 \rho_5 \mathcal{L}_1(y,p)_t , \ \rho_5 \mathcal{L}_2(y,p)_t \in L^2(Q_T),  \\
&\phantom{AAA}(\mathcal{L}_1(y,p)(0), \ \mathcal{L}_2(y,p)(0)) \in [H_0^1(\Omega)]^2, \ (y(0), \ p(0)) \in [H^3(\Omega) \cap H^1_0(\Omega)]^2 \Big\}.
\end{array}
\end{equation}

Thus, $\mathcal{Y}$ is a Hilbert space for the norm $\Vert \cdot \Vert_{\mathcal{Y}}$, where
\begin{equation*}
    \begin{array}{lll}
          \Vert (y,p,v)\Vert^{2}_{\mathcal{Y}} &=& \Vert \rho_{2}y\Vert^{2}_{L^{2}(Q_T)} + \Vert \rho_{2}p\Vert^{2}_{L^{2}(Q_T)} + \Vert\rho_{0} v\Vert^{2}_{L^{2}(\omega\times (0,T))} +\Vert\rho \mathcal{L}_1(y,p)\Vert^{2}_{L^{2}(Q_T)} \\
          && + \Vert\rho \mathcal{L}_2(y,p)\Vert^{2}_{L^{2}(Q_T)}
          +\Vert\rho_5 \mathcal{L}_1(y,p)_t\Vert^{2}_{L^{2}(Q_T)} + \Vert\rho_5 \mathcal{L}_2(y,p)_t \Vert^{2}_{L^{2}(Q_T)} \\
          &&+ \Vert y(0)\Vert^{2}_{H^{3}(\Omega)} + \Vert p(0)\Vert^{2}_{H^{3}(\Omega)}.
    \end{array}
\end{equation*}

Now, let us introduce the Banach space $\mathcal{Z} = \mathcal{F}^2  \times [H^3(\Omega) \cap H_0^1(\Omega)]^2$ such that 
$$\mathcal{F}=\Big\{ h \in L^2(Q_T) \ : \ \rho h \in L^2(Q_T), 
\ \rho_5 h_t \in L^2(Q_T), \ h(0) \in H_0^1(\Omega)
\Big\},$$
with the norm
\begin{equation*}
\begin{split}
    \|(f,g,y_0,p_0)\|_\mathcal{Z}^2 = &\|\rho f\|_{L^2(Q_T)}^2 + \|\rho g\|_{L^2(Q_T)}^2 + \| \rho_5 f_t\|_{L^2(Q_T)}^2 + \| \rho_5 g_t\|_{L^2(Q_T)}^2  \\
    &+ \|f(0)\|_{H_0^1(\Omega)} + \|g(0)\|_{H_0^1(\Omega)}+ \|y_0\|_{H^3(\Omega)} + \|p_0\|_{H^3(\Omega)}.    
\end{split}    
\end{equation*}
Finally, consider the map $\mathcal{A} : \mathcal{Y} \to \mathcal{Z}$ such that
$$\mathcal{A}(y,p,v) = (\mathcal{A}_0, \mathcal{A}_1, \mathcal{A}_2, \mathcal{A}_3)(y,p,v), \ \forall (y,p,v) \in \mathcal{Y},$$ 
where the components $\mathcal{A}_i$, $i=0,1,2,3$, are given by
\begin{equation}\label{aplicação A}
    \left\{
    \begin{array}{ll}
        \mathcal{A}_0(y,p,v) &= y_t - \nabla \cdot (\kappa(y) \nabla y) - \sigma(y) |\nabla p|^2 - 2\sigma(y) \nabla p \cdot \nabla z^*\\
        &\ \ \ \ - (\sigma(y)-\sigma(0))|\nabla z^*|^2  + v 1_\omega,\\
        \mathcal{A}_1(y,p,v) &= p_t - \nabla \cdot (\sigma(y) \nabla p) - \nabla \cdot ((\sigma(y)-\sigma(0)) \nabla z^*), \\
        \mathcal{A}_2(y,p,v) &= y(\cdot,0),\\
        \mathcal{A}_3(y,p,v) &= p(\cdot,0).
    \end{array}
    \right.
\end{equation}

\subsection{Hypotheses of Liusternik's Theorem}
We prove that we can apply Theorem \ref{Liusternik} to the mapping $\mathcal{A}$ defined in \eqref{aplicação A} through the following three lemmas:

\begin{lemma}\label{A bem definido}
    The mapping $\mathcal{A}: \mathcal{Y}\rightarrow  \mathcal{Z}$ is well defined and continuous. 
\end{lemma}

\begin{proof}
    We want to show that $\mathcal{A}(y,p,v)$ belongs to $ \mathcal{Z}$, for every $(y,p,v)\in  \mathcal{Y}$. We will therefore show that each $\mathcal{A}_{i}(y,p,v)$, with $i=\lbrace 0,1, 2, 3\rbrace$, defined in \eqref{aplicação A} belongs to its respective space. Notice that, 
    \begin{equation*}
        \begin{array}{lll}
             \displaystyle\int_{Q_T}\rho^{2}|\mathcal{A}_{0}(y,p,v)|^{2}dxdt &\leq  5\displaystyle\int_{Q_T}\rho^{2}|\mathcal{L}_1(y,p)|^{2}dxdt + 5\displaystyle\int_{Q_T} \rho^{2} |\sigma(0)|^2 |\nabla p|^4dxdt
              \\
             &\phantom{A} + 5\displaystyle\int_{Q_T}\rho^{2} \left| \nabla \cdot [(\kappa(y)-\kappa(0))\nabla y] \right|^{2}dxdt \\
             &\phantom{A}+ 5 \displaystyle\int_{Q_T} \rho^{2} |\sigma(y)-\sigma(0)-\sigma'(0)y|^2 |\nabla z^*|^4 dxdt\\
             &\phantom{A}+ 5\displaystyle\int_{Q_T} \rho^{2}  |\sigma(y)-\sigma(0)|^2 |\nabla p|^2 |\nabla z^*|^2 dxdt \\
             &=  5 I_{1} + 5 I_{2} + 5 I_3 + 5 I_4 + 5 I_5.
        \end{array}
    \end{equation*}
    
It is immediate by definition of the space $\mathcal{Y}$ that $I_{1} \leq C \|(y,p,h)\|^{2}_{\mathcal{Y}}$. Furthermore, using Propositions \ref{Prop3} and \ref{Prop4}, and the continuous embedding $H^2(\Omega) \hookrightarrow L^\infty(\Omega)$, valid for $1 \leq N \leq 3$, we obtain, for $i=2,3,4, 5$ that
\begin{equation*}
    \begin{array}{l}
        I_{i} \leq C\|(y,p,v)\|^{4}_{\mathcal{Y}}.
    \end{array}
\end{equation*}
On the other hand,
    \begin{equation*}
        \begin{array}{lll}
             \displaystyle\int_{Q_T}\rho^{2}_5|[\mathcal{A}_{0}(y,p,v)]_t|^{2}dxdt &\leq & 5\displaystyle\int_{Q_T}\rho^{2}_5 |\mathcal{L}_1(y,p)_t|^{2}dxdt + 5\displaystyle\int_{Q_T} \rho^{2}_5 | [\sigma(0) |\nabla p|^2]_t|^2 dxdt
             \\
             & &+ 5\displaystyle\int_{Q_T}\rho^{2}_5 \left|\nabla \cdot [(\kappa(y)-\kappa(0))\nabla y]_t \right|^{2}dxdt \\
             & &+ 5 \displaystyle\int_{Q_T} \rho^{2}_5 \left| [(\sigma(y)-\sigma(0)-\sigma'(0)y) |\nabla z^*|^2 ]_t \right|^2 dxdt\\
             & &+ 5\displaystyle\int_{Q_T} \rho^{2}_5 \left|[ (\sigma(y)-\sigma(0)) \nabla p \cdot \nabla z^*]_t \right|^2 dxdt \\
             & = & 5 J_{1} + 5 J_{2} + 5 J_3 + 5 J_4 + 5 J_5.
        \end{array}
    \end{equation*}
    
Standard estimates as above, using Propositions \ref{Prop3} and \ref{Prop4}, and the continuous embedding $H^2(\Omega) \hookrightarrow L^\infty(\Omega)$, give
$$
    J_1 \leq C \|(u,p,v)\|_{\mathcal{Y}}^2 \quad\text{and} \quad J_i \leq C \left(\|(u,p,v)\|_{\mathcal{Y}}^4 + \|(u,p,v)\|_{\mathcal{Y}}^6  \right), \quad \text{for } i=2,3,4,5.
$$

We also have, using the definition of the space $\mathcal{Y}$ and the continuous embedding $H^2(\Omega) \hookrightarrow L^\infty(\Omega)$, the same computation as in Lemma 3.1, item a), of \cite{Thermistor_HNLP-23},
\begin{equation*}
\begin{split}
    \|\nabla [\mathcal{A}_0(y,p,v)](0)\|^2 & \leq 3 \int_Q |\nabla\mathcal{L}_1(y,p)(0)|^2 + 3 \int_Q |\nabla \left[ \nabla \cdot (\kappa(y)-\kappa(0))\nabla y \right](0)|^2\\
    &\quad + 3 \int_Q |\nabla \left[\sigma(y)|\nabla p|^2 \right](0)|^2\\  
    &\leq C \left( \|(y,p,v)\|_{\mathcal{Y}}^2 + \|(y,p,v)\|_{\mathcal{Y}}^4 + \|(y,p,v)\|_{\mathcal{Y}}^6 \right).
\end{split}
\end{equation*}

Therefore, $\mathcal{A}_{0}(y,p,v) \in \mathcal{F}.$

Now, for the component $\mathcal{A}_{1}(y,p,v)$ we have:
\begin{equation*}
\begin{array}{lll}
    \Vert\mathcal{A}_{1}(y,p,v)\Vert^{2}_{\mathcal{F}}
    &\leq & 3\displaystyle\int_{Q_T}\rho^{2} |\mathcal{L}_2(y,p)|^{2}dxdt \\
    &\quad & + \ 3\displaystyle\int_{Q_T}\rho^{2} \left| \nabla \cdot ((\sigma(y)-\sigma(0))\nabla p) \right|^{2}dxdt \\
    &\quad & + \
    3\displaystyle\int_{Q_T}\rho^{2} \left| \nabla\cdot[ (\sigma(y)-\sigma(0)-\sigma'(0)y ) \nabla z^* ] \right|^{2}dxdt \\
    & = & 3I_{3} + 3I_{4}+3I_5.
\end{array}
\end{equation*}

By definition of the space $\mathcal{Y}$ we have that $I_{3}\leq C \|(y,p,v)\|^{2}_{\mathcal{Y}}$. Following similar computations as for the component $\mathcal{A}_0$, together with \cite{Thermistor_HNLP-23}, we get
$$
    \|\mathcal{A}_1(y,p,v)\|^2_{\mathcal{F}} \leq C \left(\|(y,p,v)\|_{\mathcal{F}}^2 + \|(y,p,v)\|_{\mathcal{F}}^4 + \|(y,p,v)\|_{\mathcal{F}}^6 \right),
$$
$$
    \|[\mathcal{A}_1(y,p,v)]_t\|^2_{\mathcal{F}} \leq C \left(\|(y,p,v)\|_{\mathcal{F}}^2 + \|(y,p,v)\|_{\mathcal{F}}^4 + \|(y,p,v)\|_{\mathcal{F}}^6 \right).
$$
and
\begin{equation*}
    \|\nabla [\mathcal{A}_1(y,p,v)](0)\|^2 \leq C \left( \|(y,p,v)\|_{\mathcal{Y}}^2 + \|(y,p,v)\|_{\mathcal{Y}}^4 + \|(y,p,v)\|_{\mathcal{Y}}^6 \right).
\end{equation*}
 
Thus $\mathcal{A}$ is well-defined.

\end{proof}

\begin{lemma}\label{DA continuo}
    The mapping $\mathcal{A}: \mathcal{Y} \rightarrow  \mathcal{Z}$ is continuously differentiable.
\end{lemma}

\begin{proof}
First we prove that $\mathcal{A}$ is Gateaux differentiable at any $(y,p,v) \in \mathcal{Y}$ and let us compute the
$\textit{G-derivative}$ ${\mathcal{A}}^{\prime}(y, p, v)$.
Consider the linear mapping $D \mathcal{A}: \mathcal{Y} \to \mathcal{Z}$ given by
$$
D\mathcal{A}(y,p,v) = (D\mathcal{A}_0,D\mathcal{A}_1,D\mathcal{A}_2,D\mathcal{A}_3),
$$
where, for $(\bar y, \bar p,\bar v) \in \mathcal{Y}$,
\begin{equation}
    \left\{
    \begin{array}{ll}
        D\mathcal{A}_0(\bar y, \bar p,\bar v) &=   \, \bar y_t - \nabla \cdot [\kappa'(y)\bar y \nabla y] - \nabla \cdot (\kappa(y) \nabla \bar y) - 2\sigma(y)\nabla p \cdot \nabla \bar p - \sigma'(y) \bar y |\nabla p|^2  \\
        &\phantom{=}- 2\sigma(y) \nabla \bar p \cdot \nabla z^* -2\sigma'(y)\bar y \nabla p \cdot \nabla z^* - \sigma'(y) \bar y |\nabla z^*|^2 - \bar v 1_\omega, \\
        D\mathcal{A}_1(\bar y, \bar p,\bar v) &= \bar p_t - \nabla \cdot (\sigma'(y) \bar y \nabla p) - \nabla \cdot (\sigma(y) \nabla \bar p) - \nabla \cdot (\sigma'(y)\bar y \nabla z^*), \\
        D\mathcal{A}_2(\bar y, \bar p,\bar v) &= \, \bar y(0),\\
        D\mathcal{A}_3(\bar y, \bar p,\bar v) &= \, \bar p(0).
    \end{array}
    \right.    
\end{equation}

We have to show that, for $i=0,1,2,3$, 
$$
\frac{1}{\lambda}\left[ \mathcal{A}_i ((y,p,v)+\lambda(\bar y, \bar p,\bar v)) - \mathcal{A}_i (y,p,v) \right] \to D\mathcal{A}_i(\bar y, \bar p,\bar v),
$$
strongly in the corresponding factor of $\mathcal{Z}$ as $\lambda \to 0$.

Indeed, using Proposition \ref{Prop3}, we have
\begin{equation*}
\begin{split}
    &\left\| \frac{1}{\lambda}\left[ \mathcal{A}_0 ((y,p,v)+\lambda(\bar y, \bar p,\bar v)) - \mathcal{A}_0(\bar y, \bar p,\bar v) \right] - D\mathcal{A}_0(y,p,v) \right\|^{2}_{L^2(\rho^2,Q_T)} \to 0,\\
\end{split}
\end{equation*}
as it is the same computation, using the Mean value theorem and the Lebesgue dominated theorem, as in Lemma 3.1 (b) of \cite{Thermistor_HNLP-23}. Here we used the notation
$\|f\|_{L^2(\rho^2,Q_T)} = \int_{Q_T} \rho^2 |f|^2$.

Analogously, Proposition \ref{Prop3} gives also 
\begin{equation}
\begin{split}
    &\left\| \frac{1}{\lambda}\left[ \mathcal{A}_1 ((y,p,v)+\lambda(\bar y, \bar p,\bar v)) - \mathcal{A}_1 (y,p,v) \right] - D\mathcal{A}_1(\bar y, \bar p,\bar v) \right\|^{2}_{L^2(\rho^2,Q_T)} \to 0, 
\end{split}
\end{equation}
as $\lambda \to 0$, since the only new term with respect to the computation in \cite{Thermistor_HNLP-23} is the term $\bar p_t$ in $D\mathcal{A}_1$ which is canceled by the term
$\displaystyle \lim_{\lambda \to 0}  \frac{1}{\lambda}\left[ (p_t - \lambda \bar p_t) - p_t \right]$.

This finish the proof that $\mathcal{A}$ is Gateaux differentiable, with a \textit{G-derivative} $\mathcal{A^{\prime}}(y,p,v)= D\mathcal{A}(y,p,v)$.

Now take $(y,p,v)\in \mathcal{Y}$ and let $((y_{n},p_{n},v_{n}))_{n \in \mathbb{N} }$ be a sequence which converges to  $(y,p,v)$ in $\mathcal{Y}$. 

Then, arguing as in \cite{Thermistor_HNLP-23} we get, for $i=0,1$, and for each $(\bar{y},\bar{p},\bar{v})\in B_{r}(0)$,
\begin{align*} 
    &\|(D\mathcal{A}_i(y_{n},p_{n},v_{n})-D\mathcal{A}_i(y,p,v))(\bar y, \bar p,\bar v)\|_{\mathcal{Z}} \\
    &\phantom{AAAAAA} \leq C \|(\bar y, \bar p,\bar v)\|_\mathcal{Y} \|(y_{n}-y),(p_{n}-p),(v_{n}-v)\|_{\mathcal{Y}}\rightarrow 0,
\end{align*}
as $n \to \infty$.
Indeed, as above the new term $\bar p_t$ in $D\mathcal{A}_1$ cancels in the difference $(D\mathcal{A}_i(y_{n},p_{n},v_{n})-D\mathcal{A}_i(y,p,v))(\bar y, \bar p,\bar v)$.

Therefore, $(y,p,v)\mapsto\mathcal{A}^{\prime}(y,p,v)$ is continuous from $\mathcal{Y}$ into $\mathcal{L}(\mathcal{Y},\mathcal{Z})$ and, in view of classical results, we have that $\mathcal{A}$ is Fr\'echet-differentiable and $\mathcal{C}^{1}$.
\end{proof}

\begin{lemma}\label{Mapa sobrejetivo}
Let $\mathcal{A}$ be the mapping in \eqref{aplicação A}. Then, $\mathcal{A}^{\prime}(0,0,0)$ is onto.
\end{lemma}
\begin{proof}
Let $(f,g, y_{0},p_0)\in \mathcal{Z}$. From Proposition \ref{Prop3} we know there exists $(y,p,v)$ satisfying \eqref{sistema2} and \eqref{eqa1-prop3}. Furthermore, from Proposition \ref{Prop4} we have that $(y,p,v)$ satisfies \eqref{eq1-Prop4}.

Consequently, $(y,p,v)\in \mathcal{Y}$ and $$\mathcal{A}^{\prime}(0,0,0)(y,p,v)=(f,g,y_{0},z_0).$$ 
This ends the proof.
\end{proof}

\subsection{Local null controllability of the nonlinear system (\ref{sistema2})} \label{subsec:control_for_nonlinear_system}

\noindent According to Lemmas \ref{A bem definido}-\ref{Mapa sobrejetivo} we can apply the Inverse Mapping Theorem (Theorem \ref{Liusternik}) and consequently there exists $\delta > 0$ and a mapping $W:B_{\delta}(0)\subset {\mathcal{Z}}\rightarrow {\mathcal{Y}}$ such that
\begin{equation*}
    W(f,g,y_0,z_0)\in B_{r}(0)\,\,\, \text{and}\,\,\, {\mathcal{A}}(W(f,g,y_0,p_0))=(f,g,y_0,p_0), \,\,\, \forall (f,g,y_0,p_0)\in B_{\delta}(0).
\end{equation*}

From hypothesis \eqref{Hyp9}-\eqref{Hyp10}, we can take $(\sigma(0)|\nabla z^*|^2, \sigma(0) \Delta z^* - (z^*)_t, y_{0}, p_0)\in B_{\delta}(0)$ and $$(y,p,v)=W(\sigma(0)|\nabla z^*|^2, \sigma(0) \Delta z^* - (z^*)_t,y_{0},p_0) \in \mathcal{Y}.$$
So, we have
\begin{equation*}
    {\mathcal{A}}(y,p,v)=(\sigma(0)|\nabla z^*|^2,\sigma(0) \Delta z^* - (z^*)_t,y_{0},p_0).
\end{equation*}
Thus, we conclude that \eqref{sistema2} is locally null controllable at time $T > 0$. This concludes the proof of Theorem \ref{maintheorem}.
\hfill\qed


\section{Additional remarks and open problems}
\label{sec:final_remarks}

\subsection{Large time null controllability of systems (\ref{sistema2}) and (\ref{sistema2_control2})}
\label{sec:large_time_control}

Let us recall the concept of \emph{large-time $R$-null controllability} from \cite{coron-07b}, see also \cite{Thermistor_HNLP-23}:

\begin{defi}
The system \eqref{sistema2} $($or \eqref{sistema2_control2}$)$ is \emph{large-time $R$-null controllable}, if for any $(y_0, p_0) \in [H^3(\Omega) \cap H_0^1(\Omega)]^2$ and $z^*$ satisfying \eqref{Hyp2}  with
$$
\|y_0\|_{H^3(\Omega)} + \|p_0\|_{H^3(\Omega)} + \|z^*\|_{H^1(0,\infty; H^3(\Omega))} < R, 
$$
there exists $T^*>0$ such that for any $T > T^*$, one can find a control $v \in L^2(\omega \times (0,T))$ such that the associated state $(y,p)$ of system \eqref{sistema2} $($or \eqref{sistema2_control2}$)$  satisfies\ \ $y(x,T)=0$  and\  $p(x,T)=0$  \text{in } $\Omega$.
\end{defi}

\begin{teo}\label{thm:large_time_result}
    Let $T_0 > 0$ be fixed. Under the assumptions \eqref{Hyp2}-\eqref{Hyp6}, there exists $r=r(\kappa_0,\kappa_1,\sigma_0,\sigma_1,M,\Omega,N)>0$ and $\overline{T} = \overline{T}(r,T_0) > 0$ such that, if
     \begin{equation}\label{Hyp7_T_large}
        z^* \in W^{2,\infty}(\Omega \times (\overline{T}, \overline{T}+T_0)),
    \end{equation}
    \begin{equation}\label{Hyp8_T_large}
        \nabla z^* \neq 0  \  \text{ in } \ \omega_0 \times (\overline{T}, \overline{T} + T_0),
    \end{equation}
    \begin{equation}\label{Hyp9_T_large}
        \rho(t- \overline{T}) |\nabla z^*|^2, \rho(t-\overline{T}) |\nabla z^*_t|^2, \rho(t-\overline{T}) \Delta z^*, \rho(t-\overline{T}) \Delta z^*_t, \rho(t-\overline{T}) z^*_t, \rho(t-\overline{T}) z^*_{tt} \in L^2(\Omega \times (\overline{T}, \overline{T} + T_0)),\  
    \end{equation}
    \begin{equation}\label{Hyp10_T_large}
        |\nabla z^*(\overline{T})|^2, \sigma(\overline{T}) \Delta z^*(\overline{T})- z^*_t(\overline{T}) \in H^1_0(\Omega).
    \end{equation}
    the system \eqref{sistema2} is large-time $R$-null controllable for $0 < R \leq r$. 
\end{teo}
\begin{proof}
By Section \ref{sec3:stability}, there exists $r=r(\kappa_0, \kappa_1, \sigma_0, \sigma_1, M, \Omega, N)>0$ such that if
$$\|y_0\|_{H^3(\Omega)} + \|p_0\|_{H^3(\Omega)} + \|z^*\|_{H^1(0,\infty; H^3(\Omega))} < R \ \text{ with }  \ 0 < R \leq r,$$
we have
$$\|y(t)\|_{H^3(\Omega)}^2 + \|p(t)\|_{H^3(\Omega)}^2 \leq C(1 + R^2 + R^4 + R^6 + R^8) e^{-\rho t}.$$
By \eqref{Hyp3}, we have that
\begin{align*}
    \|z^*\|_{H^1(t,\infty;H^3(\Omega))}^2 &= \int_t^{+\infty} \left( \|z^*(s)\|_{H^3(\Omega)}^2 + \|z^*_t(s)\|_{H^3(\Omega)}^2  \right) ds \\
            &\leq C^* \int_t^{+\infty} e^{-\gamma s} ds = \frac{C^*}{\gamma}  e^{-\gamma t},\ \forall t \geq 0 .  
\end{align*}
\begin{rem}\label{rem:shift_T}
    Fixed $T_0>0$, by Section \ref{sec4:controllability} (Liusternik's Theorem) for any $T \geq 0$, there exists $\epsilon = \epsilon(T_0)>0$, independent of $T$ such that
    $$
    \|y(T)\|_{H^3(\Omega)}^2 + \|p(T)\|_{H^3(\Omega)}^2 + \|z^*\|_{H^1(T,\infty;H^3(\Omega))}^2 < \epsilon,
    $$
    there exists a control $v_0 \in L^2(\omega \times (T, T+T_0))$ such that the associated state $(y,p)$ satisfies
    $$y(x,T+T_0)=0 \quad \text{and} \quad p(x,T+T_0)=0 \quad \text{in }\Omega.$$
\end{rem}

Taking $\overline{T} > 0$ satisfying
$$\overline{T} > \max \left\{ \frac{1}{\rho} \ln \left( \frac{4C(1 + R^2 + R^4 + R^6 + R^8)}{\epsilon^2} \right), \frac{1}{\gamma} \ln\left( \frac{4C^*}{\gamma \epsilon^2} \right) \right\},$$
the solution $(y,p)$ satisfies 
\begin{equation}\label{eq:decay_at_T}
    \|y(\overline{T})\|_{H^3(\Omega)}^2 + \|p(\overline{T})\|_{H^3(\Omega)}^2 + \| z^* \|_{H^1(\overline{T},\infty;H^3(\Omega))}^2 < \epsilon.
\end{equation}

Denoting $T^* = \overline{T} + T_0$, by Remark \ref{rem:shift_T} with $T=\overline{T}$ and \eqref{Hyp7_T_large}-\eqref{eq:decay_at_T}, there exists a control $v^* \in L^2(\omega \times (T,T^*))$ such that
$$y(x,T^*) = 0  \ \text{ and }\  p(x,T^*)=0 \  \text{ in } \Omega.$$
Then, for any $T\geq T^*$, we can define the control $v$ such that
\begin{equation*}
    v(x,t) = \begin{cases}
        0, & \text{if}\quad 0 \leq t \leq \bar T,\\
        v^*(x,t), & \text{if}\quad \bar T \leq t \leq T^*,\\
        0, &\text{if}\quad T^* \leq t \leq T.
    \end{cases}
\end{equation*}
\end{proof}
\subsection{Further questions}

Finally, we present some open questions.

\begin{itemize}
    \item The local null controllability in the case of $N \geq 4$ is open in the case parabolic-parabolic studied in this paper but also in the case parabolic-elliptic treated in \cite{Thermistor_HNLP-23}. The main problem comes from the Sobolev immersion $H^2(\Omega) \hookrightarrow L^\infty(\Omega)$ used in the nonlinear estimates, which is valid for $1 \leq N \leq 3$.

    \item  In \cite{HMRR-10} a two-dimensional parabolic-elliptic thermistor problem is studied with mixed nonlinear boundary conditions and control acting on the boundary. Optimal control of this problem was shown in \cite{Hannes-existence-17} and \cite{Hannes-optimality-17} where maximal bounds on the temperature and on the potential were considered in order to perform a realistic modeling of the process.
    The controllability of problems \eqref{sistema2} and \eqref{sistema2_control2} with mixed-boundary conditions is an open problem.

    \item The system \eqref{sistema_thermistor} admits a hierarchical control formulation. In this situation there are follower controls looking to minimize some functional and a leader control which leads the state to zero. The local hierarchical controllability of the system \eqref{sistema_thermistor} when the followers admit a non-cooperative (or Nash) equilibrium is an ongoing project of the authors.
       
\end{itemize}

\textbf{Acknowledgments}

This study was financed in part by the Coordenação de Aperfeiçoamento de Pessoal de Nível Superior - Brasil (CAPES) - Finance Code 001.
L.Y. was partially supported by CAPES-Brazil. J.L. was partially supported by CNPq-Brazil


\end{document}